\documentclass[11pt,a4paper]{amsart}
\usepackage{amsmath}
\usepackage{amsfonts}
\usepackage{graphicx}
\usepackage{framed}
\usepackage{amssymb,cancel}
\usepackage{xcolor}
\usepackage{paralist}
\usepackage{etex}
\usepackage{latexsym}
\usepackage[english]{babel}
\usepackage[utf8x]{inputenc}

\def \RN {\mathbb{R}^N}
\def \G {\mathbb{G}}

\def \R {\mathbb{R}}

\def \de {\partial}

\def \LL {\mathcal{L}}

\def \LieG {\mathrm{Lie}(\mathbb{G})}

\usepackage{amsthm}
\theoremstyle{definition}
\newtheorem{definition}{Definition}[section]
\newtheorem*{axiom}{Assumptions}
\newtheorem{example}[definition]{Example}
\newtheorem{remark}[definition]{Remark}

\theoremstyle{plain}
\newtheorem{theorem}[definition]{Theorem}
\newtheorem{proposition}[definition]{Proposition}
\newtheorem{lemma}[definition]{Lemma}

\numberwithin{equation}{section}
\usepackage[colorlinks=true,urlcolor=blue,
citecolor=red,linkcolor=blue,linktocpage,pdfpagelabels,
bookmarksnumbered,bookmarksopen]{hyperref}
\usepackage[hyperpageref]{backref}

\begin{document}

 \title[A Liouville property for H\"ormander operators] {A Liouville-type property for
degenerate-elliptic equations modeled on H\"ormander vector fields}
 \author[S.\,Biagi]{Stefano Biagi}
 \author[D.D.\,Monticelli]{Dario Daniele Monticelli}
 \author[F.\,Punzo]{Fabio Punzo}

 \address[S.\,Biagi]{Dipartimento di Matematica
 \newline\indent Politecnico di Milano \newline\indent
 Via Bonardi 9, 20133 Milano, Italy}
 \email{stefano.biagi@polimi.it}

 \address[D.D.\,Monticelli]{Dipartimento di Matematica
 \newline\indent Politecnico di Milano \newline\indent
 Via Bonardi 9, 20133 Milano, Italy}
 \email{dario.monticelli@polimi.it}

  \address[F.\,Punzo]{Dipartimento di Matematica
 \newline\indent Politecnico di Milano \newline\indent
 Via Bonardi 9, 20133 Milano, Italy}
 \email{fabio.punzo@polimi.it}



\thanks{All authors are member of the ``Gruppo Nazionale per l'Analisi Ma\-te\-ma\-tica, la Probabilit\`a e le loro 
Applicazioni'' (GNAMPA) of the ``Istituto Nazionale di Alta Matematica'' (INdAM, Italy).
The first author is partially supported by the PRIN 2022 project 2022R537CS \emph{$NO^3$ - Nodal Optimization, 
NOnlinear elliptic equations, NOnlocal geometric problems, with a focus on regularity}, founded by the European 
Union - Next Generation EU. The second and the third authors are 
partially supported by the
PRIN projects 2022 \emph{Geometric-analytic methods for PDEs
and applications}, ref.\,2022SLTHCE, financially supported by the EU, in the
framework of the ``Next Generation EU initiative''.}
 \begin{abstract}
We obtain Liouville type theorems for degenerate elliptic equation with a drift term and a potential. The diffusion is driven by H\"ormander operators. We show that the conditions imposed on the coefficients of the operator are optimal. Indeed, when they fail we prove that infinitely many bounded solutions exist. 
  \end{abstract}
 \maketitle
\section{Introduction}\label{sec.Intro}
The study of degenerate-elliptic Schrödinger-type equations has gained considerable attention in recent years due to their intricate structural properties and their connections to various fields of analysis and geometry. In this paper, we focus on the equation
\begin{equation}\tag{{\it $E$}} \label{eq:MainPDE}
\LL u = \sum_{i = 1}^m X_i^2 u +
\sum_{i = 1}^m b_i(x)X_i u - Q(x)u = 0 \quad \text{in $\R^n$},
\end{equation}
where $X = \{X_1,\ldots,X_m\}$ is a family of smooth vector fields satisfying the Hörmander rank condition,
\emph{which are not necessarily
left\,-\,invariant with respect to any Lie\,-\,gro\-up structure on $\R^n$}. 
The coefficients $b_i(x)$ and the potential function $Q(x)$ are assumed to satisfy suitable regularity and positivity conditions.

Our main objective is to establish optimal conditions on the vector fields and the potential $Q$ that ensure the validity of the Liouville property for solutions of \eqref{eq:MainPDE}. More precisely, we aim to determine when every bounded, smooth solution of \eqref{eq:MainPDE} must necessarily be trivial, i.e.,
\begin{equation}\tag{{\it $LP$}} \label{eq:LPprop}
\text{if $u\in C^\infty(\R^n)$ is a bounded solution of \eqref{eq:MainPDE}, then $u \equiv 0$.}
\end{equation}
As it frequently happens when \emph{global properties} are involved,
the question of the Liouville property for degenerate operators has been mainly studied
in the context of Lie groups, under the assumption $Q \equiv 0$, see, e.g., 
\cite{BonfKog, Kog, KPP, KL}. The case of (not necessarily invariant) H\"ormander operators, like
the ones considered here, is not as well studied; in this direction,
we mention e.g., the papers \cite{KoLa, LaKo}, where the Authors consider again the case $Q\equiv 0$.
We stress that, in all the aforementioned papers,
 the Authors employ techniques different from those considered here. 
\medskip
 
 For nonzero potential functions, a Liouville theorem for the equation  on a complete, noncompact Riemannian manifold was obtained in \cite{Grig1}. Inspired by this approach, the work in \cite{BP1} considered the equation
\[\Delta u + \langle b(x), \nabla u \rangle - Q(x) u  = 0 \]
in a bounded, connected open set $\Omega \subseteq \R^n$. This result was further generalized in \cite{BP2} to the setting of divergence-form operators:
\begin{equation*}
\mathrm{div}\big\{A(x)\nabla u\big\} - Q(x)u = 0 \quad \text{in $\Omega$},
\end{equation*}
where $A(x)$ is a symmetric, uniformly elliptic matrix with $C^2$ entries, and $Q(x)$ is a non-negative, locally Hölder-continuous function.

In this paper, we extend the results of \cite{Grig1, BP1, BP2} to the degenerate-elliptic setting of \eqref{eq:MainPDE}, where the operators involved exhibit a highly non-Euclidean structure. This generalization presents new challenges, particularly due to the presence of the vector fields $X_i$ satisfying the Hörmander condition, which induce a sub-Riemannian geometry. We develop novel techniques tailored to this setting, addressing the fundamental difficulties arising from the degeneracy of the operator and the interplay between the drift terms and the potential function.

 \subsection{Statement of the main result}
  We now turn to state and describe our main result.
 To this end, we
 fix here the main assumptions we require on the vector fields
 $X_1,\ldots,X_m$, and we refer to Section \ref{sec:Preli} for some related comments.
 \vspace*{0.1cm}

Throughout what follows, we set
 $$\LL_0u = \sum_{i = 1}^mX_j^2u+\sum_{i = 1}^m b_i(x)X_iu,$$
 and we refer to the vector field $Y = \sum_{i = 1}^m b_i(x)X_i$ as the \emph{drift of $\LL$}.
 \vspace*{0.1cm}

 \noindent\textbf{Notation.} In what follows, we will exploit the
 notation listed below.
 \vspace*{0.1cm}

 \begin{compactenum}[i)]
 \item $\mathcal{X}(\R^n)$ denotes the \emph{Lie algebra of the smooth
 vector fields} in $\R^n$.
 \vspace*{0.05cm}

 \item If $A\subseteq\mathcal{X}(\R^n)$, $\mathrm{Lie}(A)$ denotes the
 \emph{smallest Lie sub-algebra of $\mathcal{X}(\R^n)$};
 \vspace*{0.05cm}

 \item If $X = \sum_{i = 1}^na_i(x)\de_{x_i}\in\mathcal{X}(\R^n)$ and $x\in\R^n$ is fixed, we set
 $$X_x = \begin{pmatrix}
  a_1(x)\\
  \vdots\\
   a_n(x)
   \end{pmatrix}\in\R^n;$$
 \item Given any $X\in\mathcal{X}(\R^n)$, we denote by $X^*$
 the formal adjoint of $X$ (with respect
 to the usual scalar product in $L^2(\R^n)$), that is,
 \begin{equation} \label{eq:defXadjoint}
  u\mapsto X^*u = -Xu-\mathrm{div}(X_x)u.
 \end{equation}
 By definition, $X^*$ is the unique first-order operator such that
 \begin{equation*}
   \int_{\R^n}u\,Xv\,dx = \int_{\R^n}v\,X^*u\,dx,
 \end{equation*}
for every $u,v\in\mathrm{Lip}_{\mathrm{loc}}(\R^n)$, provided that at least one of the two functions has
compact sup\-port.
Note that $X^*$ is not necessarily a vector field, unless
$$\mathrm{div}(X_x) \equiv 0.$$

 \item If $X = \{X_1,\ldots,X_m\}\subseteq\mathcal{X}(\R^n)$ (with $m\geq 2$), we set
 \begin{equation} \label{eq:nablaXdivX}
 \begin{split}
  \ast)\,\,&\nabla_Xf = (X_1f,\ldots,X_mf)\quad\forall\,\,f\in C^\infty(\R^n); \\
  \ast)\,\,&\mathrm{div}_X(F) = \textstyle\sum_{i = 1}^m X_iF_i\quad\forall\,\,F = (F_1,\ldots,F_m)
  \in C^\infty(\R^n;\R^m).
 \end{split}
 \end{equation}
 We explicitly point out that, as in the classical case when $X_i = \de_{x_i}$,
 the above differential operators are related by the following formulas
 \begin{equation} \label{eq:relationdivnablaX}
 \begin{split}
 \mathrm{i)}\,\,&\textstyle\mathrm{div}_X(\nabla_X u) = \sum_{i = 1}^m X_i^2u, \\
 \mathrm{ii)}\,\,&\mathrm{div}_X(uF) =
 \langle\nabla_X u,F\rangle+u\,\mathrm{div}_X(F),
 \end{split}
 \end{equation}
 holding true for every $u\in C^\infty(\R^n)$ and every $F = (F_1,\ldots,F_m)\in C^\infty(\R^n;\R^m)$.
 Furthermore, owing to \eqref{eq:nablaXdivX} we can re-write $\LL_0$ as
 $$\LL_0 u = \mathrm{div}_X(\nabla_X u)+\langle\mathbf{b}(x),\nabla_X u\rangle,$$
 where $\mathbf{b}$ is the vector-valued function $\mathbf{b} = (b_1,\ldots,b_m)$.
 \end{compactenum}
 \begin{axiom} We assume that $X = \{X_1,\ldots,X_m\}$ (with $m\geq 2$)
 is a family of \emph{smooth vector fields in $\R^n$} (that is,
 $X_1,\ldots,X_m\in\mathcal{X}(\R^n)$) such that
 \vspace*{0.05cm}

 \begin{compactenum}
 \item[$(H1)$] $X_1,\ldots,X_m$ are homogeneous of degree $1$
 with respect to a family of \emph{non-iso\-tropic dilations}
 $\{\delta_\lambda\}_{\lambda}$ in $\R^n$ of the following form
 \begin{equation*} 
\delta_{\lambda}(x):=(\lambda^{\sigma_{1}}x_{1}%
,\ldots,\lambda^{\sigma_{n}}x_{n}),
\end{equation*}
where $\sigma = (\sigma_{1},\ldots,\sigma_{n})\in\mathbb{N}^n$ and
$1=\sigma_{1}\leq\ldots\leq\sigma_{n}$.
 This means, precisely, that for every $f\in C^\infty(\R^n),\,\lambda > 0$ one has
 $$X_i(f\circ\delta_\lambda) = \lambda (X_if)\circ\delta_\lambda
 \quad (i = 1,\ldots,m).$$
  \item[$(H2)$] $X_1,\ldots,X_m$ satisfy the \emph{H\"ormander rank condition at $x=0$}, that is,
 \begin{equation*}
  \mathrm{dim}\big\{Z_0:\,Z\in\mathrm{Lie}(X)\big\}  =n,
 \end{equation*}
 \end{compactenum}
 \vspace*{0.1cm}

 \end{axiom}
 Let $\mathcal{N}:\R^n\rightarrow[0,+\infty)$ be an exhaustion function with compact sublevels, which is $C^1$ outside some compact set $K\subset\R^n$.
 Given any $r > 0$ large enough, we set
 \begin{equation} \label{eq:defFactorS}
 \mathcal{S}(r) = \int_{\{\mathcal{N}(x) = r\}}
 \frac{|\nabla_X\mathcal{N}|^2}{|\nabla\mathcal{N}|}\,d\mathcal{H}^{n-1}.
 \end{equation}
 Note that $$\nabla_X\mathcal{N}(x)=S(x)^T\nabla\mathcal{N}(x)$$ at every point $x\in\R^n\setminus K$, for a suitable $n\times m$ matrix $S(x)$, see \eqref{1}. Therefore we have 
 $$0\leq \frac{|\nabla_X\mathcal{N}(x)|^2}{|\nabla\mathcal{N}(x)|}\leq |S(x)|^2|\nabla\mathcal{N}(x)|$$ and thus with an abuse of notation we agree to set $$\frac{|\nabla_X\mathcal{N}|^2}{|\nabla\mathcal{N}|}=0$$ at all points where $|\nabla\mathcal{N}|=0$.  We can now state our main result.
 \begin{theorem} \label{thm:MainThm}
  Let $X = \{X_1,\ldots,X_m\}\subseteq\mathcal{X}(\R^n)$ satisfy assumptions $(H1)$\,-\,$(H2)$.
  Moreover, let $\mathbf{b}\in C^\infty(\R^n,\R^m)$, $Q\in C^\infty(\R^n)$ and let $\mathcal{N}:\R^n\rightarrow[0,+\infty)$ be an exhaustion function with compact sublevels which is $C^1$ outside some compact set $K\subset\R^n$.
  We suppose that
  \vspace*{0.05cm}

 \begin{compactenum}
 \item[$(S)$] $Q\geq 0$ and $Q\not\equiv0$ pointwise in $\R^n$;
 \vspace*{0.1cm}

 \item[$(G)$] there exist constants
 $\kappa,\,\rho_0 > 0$ and a function $\hat{q}:(\rho_0,+\infty)\to[0,+\infty)$,
 \emph{co\-ntinuous and non identically $0$ on $\mathcal{I}_0 = (\rho_0, +\infty)$}, such that
 \vspace*{0.1cm}

 \noindent \emph{i)}\,\,for every $x\in\R^n$ with $\mathcal{N}(x) > \rho_0$, we have
 \begin{align}
   & Q(x)\geq |\nabla_X \mathcal{N}(x)|^2\hat{q}(\mathcal{N}(x)); \label{eq:assQgeq}\\
   & |\mathbf{b}(x)|^2+(\mathrm{div}_X(\mathbf{b}))_-
    \leq \kappa\,|\nabla_X \mathcal{N}(x)|^2\Big(\int_{\rho_0}^{\mathcal{N}(x)}
    \!\!\!\sqrt{\hat{q}(t)}\,dt\Big)^2\hat{q}(\mathcal{N}(x));  \label{eq:controllobLontano}
  \end{align}

  \noindent\emph{ii)}\,\,for every $x\in\R^n$ with $\mathcal{N}(x) \leq \rho_0$, we have
 \begin{equation} \label{eq:controllobVicino}
   |\mathbf{b}(x)|^2+(\mathrm{div}_X(\mathbf{b}))_-
   \leq \kappa Q(x).
  \end{equation}
 \end{compactenum}
 Here, $\mathbf{b} = (b_1,\ldots,b_m)$ and
  $(\mathrm{div}_X(\mathbf{b}))_-$ is the negative part of $\mathrm{div}_X(\mathbf{b})$
  \emph{(}as defined in the above notation \eqref{eq:nablaXdivX}\emph{)}, that is,
  $\mathrm{div}_X(\mathbf{b})_- = \max\{-\mathrm{div}_X(\mathbf{b}),0\}.$
 \vspace*{0.05cm}

  Then, if there exists a constant $\Lambda > 0$ such that
 \begin{equation} \label{eq:mainAssumptionIntegral}
   \int_{\rho_0}^{+\infty}\frac{1}{\mathcal{S}(r)}\exp\Big\{\Lambda
   \Big(\int_{\rho_0}^r\sqrt{\hat{q}(s)}\,ds\Big)^2\Big\}\,dr = +\infty,
 \end{equation}
 the Liouville property \eqref{eq:LPprop} holds for equation \eqref{eq:MainPDE}.
 \end{theorem}

 \begin{remark}
   Note that if in Theorem \ref{thm:MainThm} we let $Q\equiv0$ in $\R^n$, then by condition $(G)$ we immediately derive that it must also be $\mathbf{b}\equiv0$ in $\R^n$. Hence equation \eqref{eq:MainPDE} reduces to $$\sum_{i=1}^mX_i^2u=0,$$ for which the Liouville property is known to hold, see e.g. \cite{KoLa}.
 \end{remark}
Let us present the structure of the present paper. In Section \ref{sec:Preli} 
some basic pro\-per\-ties of the operator $\mathcal L$ are discussed, and some auxiliary results are proved. The main result, Theorem \ref{thm:MainThm}, is shown in Section \ref{proof}. In Section \ref{sec:CarnotGroups} the main result is specialized to the case of Carnot groups, the relevant definitions
being recalled in the Appendix. Then in Section \ref{examples}  we provide some examples
 in which our main result is applied. For the sake of concreteness, these examples will be given in the context
 of \emph{Carnot groups} again. 
Then the optimality of the hypotheses in Theorem \ref{thm:MainThm} will be discussed in Section \ref{sec:Optimality} below.

 \section{Preliminaries} \label{sec:Preli}
 In this first section we collect some preliminary material needed for the proof
 of our main results. To begin with, we  present
 some distinguished properties of $\LL$ which follow
 from assumptions $(H1)$\,-\,$(H2)$; then,
 we establish some results (of independent interest)
 allowing to simplify the proof of Theorem \ref{thm:MainThm}.
 \vspace*{0.1cm}

 To ease the readability
 of the subsequent statements,
 from now on (and throughout the rest of the paper) we tacitly understand that
 \begin{itemize}
  \item[1)] $X =
 \{X_1,\ldots,X_m\}$ is a set of smooth vector fields
 in $\R^n$ satisfying assumptions $(H1)$\,-\,$(H2)$
 (of which we inherit the notation);
 \vspace*{0.05cm}

 \item[2)] $\LL$ is as in \eqref{eq:MainPDE}, with $\mathbf{b}\in C^\infty(\R^n,\R^m)$ and
 $Q\in C^\infty(\R^n),\,Q\geq 0$ in $\R^n$.
 \end{itemize}
 Moreover, if $\mathcal{O}\subseteq\R^n$ is an arbitrary open set and $k \in\mathbb{N}\cup\{\infty\}$, we define
 $$C^k_+(\mathcal{O}) = \{f\in C^k(\mathcal{O}):\,\text{$f\geq 0$ pointwise in $\mathcal{O}$}\}.$$
 \subsection*{Some relevant properties of $\LL$}
 As anticipated, we begin this section
 by listing some distinguished properties of $\LL$
 which follow from assumptions $(H1)$\,-\,$(H2)$; we will repeatedly exploit these properties
 in the sequel.
 \vspace*{0.1cm}

 To begin with we observe that, by combining
 $(H1)$\,-\,$(H2)$, it is not difficult to recognize
	that $X_1,\ldots,X_m$ satisfy H\"ormander's Rank Condition
	not only at the origin, but
	at every point $x\in\R^n$ (see, e.g., \cite[Remark 6]{BBLondon}), that is,
	\begin{equation} \label{eq.Hormanderovunque}
	\mathrm{dim}_\R\big(\big\{Z_x:\,Z\in\mathrm{Lie}(X)\big\}\big)
   = n, \qquad\text{for every $x\in\R^n$};
   \end{equation}
   hence, $\LL$ falls in the class of the so-called \emph{H\"ormander operators}.
   As a consequence of this fact, we infer the following
   key properties.
   \begin{enumerate}[1)]
   \item By the celebrated H\"ormander's Hypoellipticity
   Theorem \cite{Hormander}, the operator
   \emph{$\LL$ is $C^\infty$\--hy\-poelliptic
   in every open subset of $\R^n$}; this means, precisely, that
   for every open set $\Omega\subseteq\R^n$ and every $f\in C^\infty(\Omega)$ we have
   $$\text{$\LL u = f$ in $\mathcal{D}'(\Omega)\,\,\Longrightarrow\,\,u\in C^\infty(\Omega)$}.$$
   In particular, if $u$ is any distributional solution of \eqref{eq:MainPDE}, then
   $u\in C^\infty(\R^n)$.
   \vspace*{0.1cm}

   \item The operator $\LL$ satisfies
   the following Strong Ma\-xi\-mum Principle
   (SMP):
    \emph{as\-su\-me that $\Omega\subseteq\R^n$
   is a
   con\-nec\-ted open set, and let
    $u\in C^2(\Omega)$ be such that
   $\LL u
   \geq 0$ on $\Omega$; if $u$ attains a non-negative maximum in $\Omega$, then $u$ is
   constant in $\Omega$} (see, e.g.,
   \cite[Corollary 3.1]{Bony} for a proof).
   \vspace*{0.1cm}

   \item The operator $\LL$ also satisfies the
   fol\-low\-ing Weak Maximum Principle (WMP):
   \emph{let $\Omega\subseteq\R^n$
   be a bounded open set, and let $u\in C^2(\Omega)\cap C(\overline{\Omega})$; then,}
   $$\begin{cases}
    \LL u\geq 0 & \text{in $\Omega$} \\
    u\leq 0 & \text{on $\de\Omega$}
   \end{cases}\,\,\Longrightarrow\,\,\text{$u\leq 0$
   throughout $\Omega$}
   $$
   (see, e.g., \cite[Proposition 10.16]{BBLibro} for a proof).
   \vspace*{0.1cm}

   \item The operator $\LL$ is \emph{non-totally degenerate} (NTD), that is,
   $$A(x)\neq 0\quad\text{for every $x\in\R^n$},$$
   where $A(x)$ is the \emph{principal matrix of $\LL$}.
   In fact, by direct computation we see that
   the matrix $A(x)$ takes the following explicit expression
   \begin{equation}\label{1}
    \begin{gathered}
     A(x) = S(x)\cdot S(x)^T, \\\text{where
   $S(x) = \big((X_1)_x\cdots (X_m)_x\big)\in M_{n,m}(\R)$}
   \end{gathered}
   \end{equation}
   (since we can write $\LL u = \mathrm{div}(A(x)\nabla u)+Wu-Q(x)u$ for some $W\in\mathcal{X}(\R^n)$);
   on the other hand, since \eqref{eq.Hormanderovunque} holds,
   for every fixed $x\in\R^n$ there exists
   $1\leq k_x\leq m$ such that
   $$(X_{k_x})_x\neq 0,$$
   and thus $A(x) = S(x)\cdot S(x)^T\neq 0$, as claimed.
    \end{enumerate}
 We also point out that,
 since the vector fields $X_1,\ldots,X_m$ satisfy condition $(H1)$, for every $i=1,\ldots,m$ we have
 $\operatorname{div}((X_i)_x)\equiv 0$ (here, $\operatorname{div}$
 denotes the usual divergence operator in $\R^n$),
 see, e.g., \cite[Section 1.5]{BLUlibro}; hence,
 by \eqref{eq:defXadjoint} we have
 \[
 X_i^*=-X_i,
 \]
 and the following integration by parts formula holds
 \begin{equation}\label{eq:intbyparts}
 \int_{\mathbb{R}^n}vX_iu\,dx=-\int_{\mathbb{R}^n}uX_iv\,dx
 \end{equation}
 for every $u,v \in
 \mathrm{Lip}_{\mathrm{loc}}(\mathbb{R}^n)$, if at least one of the two functions has compact support.

\begin{remark} \label{rem:PropertiesValgonoGenerali}
  We explicitly remark that the above properties hold for any H\"{o}rmander operator of the form $$\mathcal{P}=\sum_{i=1}^mX_i^2+Z-c$$ with $Z\in\mathcal{X}(\R^n)$ and $c\in C^\infty(\R^n)$, $c\geq0$.
\end{remark}

 \subsection*{Auxiliary results}
 Now that we have recalled the most relevant properties
 of $\LL$,
 we proceed by establishing some results (of independent interest)
 which will be used (as key ingredients)
 in the proof of Theorems \ref{thm:MainThm}\,-\,\ref{thm:Optimal}.
 \medskip

 The first result we aim to prove is the following
 \begin{proposition} \label{prop:Solnonneg}
   Assume that $u\in C^2(\R^n)$ is a
   (non-trivial) bounded solution of \eqref{eq:MainPDE}.
   Then, there exists a solution $\hat{u}\in C^2(\R^n)$ of \eqref{eq:MainPDE},
   further satisfying
   $$\text{$0<\hat{u}\leq 1$ pointwise in $\R^n$}.$$
 \end{proposition}
 To prove Proposition \ref{prop:Solnonneg}, we first establish the following lemmas.
 \begin{lemma}\label{lem:ExistenceOmegazero}
  There exists a bounded, open and convex neighborhood $\Omega_1\subseteq\R^n$ of the origin
  with the following property:
  for every $Z\in\mathcal{X}(\R^n)$, every $c\in C_+^\infty(\R^n)$,
  e\-very $f\in C(\overline{\Omega}_1)$ and every $\varphi\in C(\de\Omega_1)$,
  there exists a \emph{unique}
  $u\in C(\overline{\Omega}_1)$ s.t.
  \begin{equation}\label{2}
  \begin{cases}
  \mathcal{P}u = \displaystyle\sum_{i = 1}^mX_i^2u+Zu-c(x)u = f & \text{in $\mathcal{D}'(\Omega_1)$} \\
  u = \varphi & \text{pointwise on $\de\Omega_1$}.
  \end{cases}
  \end{equation}
 \end{lemma}
 \begin{remark}
 As it will be clear for the proof, in Lemma \ref{lem:ExistenceOmegazero}
 it is enough to assume  that the vector field $Z$ and the functions $c$
 are defined (and smooth) on an open neighborhood of the origin; we also point out that
  the neighborhood $\Omega_1$ constructed in Lemma \ref{lem:ExistenceOmegazero} depends only on the vector fields $X_1,\ldots,X_m$.
 \end{remark}
 \begin{proof}
We split the proof in two steps.

\medskip

\textit{Step 1.} We start with the case when $c>0$. Since $X_1,\ldots,X_m$ satisfy the H\"{o}rmander condition, the matrix of the principal part of the operator $\mathcal{P}$, given by \eqref{1}, is non-totally degenerate. Hence there exist $\nu_0\in\mathbb{R}^n$ with $|\nu_0|=1$, $\varepsilon,\delta>0$ and a neighborhood $\mathcal{V}$ of the origin such that
\begin{equation}\label{3}
\langle A(x)\nu,\nu\rangle\geq\delta>0
\end{equation}
for all $x\in\mathcal{V}$ and all $\nu\in\mathbb{R}^n$ with $|\nu|=1$ and $|\nu-\nu_0|<\varepsilon$. Then following \cite[Corollary 5.2]{Bony} we can define
\[
\Omega_1=B_{M+\eta}(M\nu_0)\cap B_{M+\eta}(-M\nu_0)
\]
for a suitably small $\eta>0$ and a suitably large $M>0$, so that $\Omega_1\subset\mathcal{V}$ and it satisfies the \emph{non-characteristic exterior ball condition} at each point of its boundary. In other terms, for every $x\in\partial\Omega_1$ there exists a ball $B=B_r(\xi)$ such that
 \[
 \overline{B}\cap\overline{\Omega}_1=\{x\}\quad\text{ and }\quad\langle A(x)(\xi-x),(\xi-x)\rangle>0.
\]
Therefore by \cite[Theoerem 5.2]{Bony} we can conclude that
for every $Z\in\mathcal{X}(\R^n)$, every $c\in C^\infty(\R^n)$ with $c>0$,
  e\-very $f\in C(\overline{\Omega}_1)$ and every $\varphi\in C(\de\Omega_1)$, there exists a unique solution
  $u\in C(\overline{\Omega}_1)$ of \eqref{2}.

\medskip

\textit{Step 2.} Now we consider the general case when $c\geq0$. Let $\Omega_1$ be the neighborhood of the origin constructed in Step 1. Let $Z$, $c$, $f$ and $\varphi$ be as in the statement of the Lemma; we set
\[
\psi(x)=K-e^{-\alpha\langle x,\nu_0\rangle}
\]
for suitable $K,\alpha>0$ to be chosen later such that $\psi>0$ on $\overline{\Omega}_1$, and  we define
\[
\mathcal{Q}u=\frac{1}{\psi}\mathcal{P}(\psi u).
\]
Now it is easy to see that
\[
\mathcal{Q}u=\displaystyle\sum_{i = 1}^mX_i^2u+\widetilde{Z}u-\widetilde{c}(x)u
\]
with $\widetilde{Z}\in\mathcal{X}(\mathbb{R}^n)$, $\widetilde{c}\in C^\infty(\mathbb{R}^n)$ given by
\[
\widetilde{Z}u=Zu+\frac{2}{\psi}\sum_{i=1}^mX_i\psi X_iu,\qquad \widetilde{c}=c-\frac{1}{\psi}\left(\sum_{i=1}^mX^2_i\psi+Z\psi\right).
\]
Moreover a simple computation using \eqref{3} shows that
\[
\begin{aligned}
   \widetilde{c}(x)& \geq-\frac{1}{\psi(x)}\left(\sum_{i=1}^mX^2_i\psi(x)+Z\psi(x)\right)\\
                &=\frac{e^{-\alpha\langle x,\nu_0\rangle}}{K-e^{-\alpha\langle x,\nu_0\rangle}}\left(\alpha^2\langle A(x)\nu_0,\nu_0\rangle-\alpha\langle Z_x,\nu_0\rangle\right)\\
                &\geq \frac{e^{-\alpha\langle x,\nu_0\rangle}}{K-e^{-\alpha\langle x,\nu_0\rangle}}\left(\delta\alpha^2-\alpha \|Z_x\|_{L^\infty(\Omega_1)}\right)>0,
\end{aligned}
\]
provided that $\alpha,K>0$ are chosen sufficiently large.

By Step 1 there exists a unique solution $v\in C(\overline{\Omega}_1)$ of
\begin{equation*}
  \begin{cases}
  \mathcal{Q}v =\frac{f}{\psi} & \text{in $\mathcal{D}'(\Omega_1)$} \\
  v = \frac{\varphi}{\psi} & \text{pointwise on $\de\Omega_1$}.
  \end{cases}
  \end{equation*}
Then $u=\psi v\in C(\overline{\Omega}_1)$ is clearly a solution of problem \eqref{2}, and it is unique.
\end{proof}

 \begin{lemma} \label{lem:SequenceInvading}
  There exists
   a sequence $\{\mathcal{V}_k\}_k\subseteq\R^n$ of bounded, open and convex
   neighborhoods of the origin with the following properties:
   \begin{itemize}
    \item[\emph{1)}]
    $\mathcal{V}_k\Subset\mathcal{V}_{k+1}$ and $\bigcup_{k = 1}^{\infty}\mathcal{V}_k = \R^n$;
    \vspace*{0.1cm}

    \item[\emph{2)}]
    given any $k\geq 1$, for every
    $Z\in\mathcal{X}(\R^n)$, every
    $c\in C^\infty_+(\R^n)$, every $f\in C(\overline{\mathcal{V}}_k)$ and every $\varphi\in C(\de\mathcal{V}_k)$
  there exits a \emph{u\-ni\-que} function $u\in C(\overline{\mathcal{V}}_k)$ such that
  $$\begin{cases}
  \mathcal{P} u = \textstyle\sum_{i = 1}^mX_i^2u+Zu-c(x)u =  f & \text{in $\mathcal{D}'(\mathcal{V}_k)$} \\
  u = \varphi & \text{pointwise on $\de\mathcal{V}_k$}.
  \end{cases}$$
   \end{itemize}
 \end{lemma}
\begin{proof}
  For every $\lambda>0$ we define $\Omega_\lambda=\delta_\lambda(\Omega_1)$, where $\Omega_1$ is the neighborhood of the origin constructed in Lemma \ref{lem:ExistenceOmegazero}. Since every $\Omega_\lambda$ is a bounded, open and convex neighborhood of $0$, and since
  \[
  \bigcup_{\lambda> 1}^{\infty}\Omega_\lambda = \R^n,
  \]
  we can choose a sequence $\lambda_k$ monotonically diverging to $+\infty$ such that the family of sets $\{\mathcal{V}_k\}_k=\{\Omega_{\lambda_k}\}_k$ satisfies property 1) in the statement.
\vspace*{0.1cm}

  We are now going to show that
  the family $\{\mathcal{V}_k\}_k$ also satisfies property 2). To this end,
  we fix $\lambda>0$, $Z = \sum_{j = 1}^n b_j(x)\partial_{j}\in\mathcal{X}(\mathbb{R}^n)$,
   $c\in C^\infty_+(\mathbb{R}^n)$, $f\in C(\overline{\Omega}_\lambda)$ and $\varphi\in C(\partial\Omega_\lambda)$.
   We can easily see that, since $X_1,\ldots,X_m$ satisfy assumption $(H1)$, a function $u\in C(\overline{\Omega}_\lambda)$
   solves the problem
   \begin{equation} \label{eq:PbOmegaLambdau}
    \begin{cases}
  \mathcal{P} u = f & \text{in $\mathcal{D}'(\Omega_\lambda)$} \\
  u = \varphi & \text{pointwise on $\de\Omega_\lambda$}
  \end{cases}
  \end{equation}
  if and only if the function $v = u\circ\delta_{\lambda}\in C(\overline{\Omega}_1)$ solves
  \begin{equation}  \label{eq:PbOmegav}
   \begin{cases}
  \mathcal{Q} v = \lambda^2 f\circ\delta_{\lambda} & \text{in $\mathcal{D}'(\Omega_1)$} \\
  v = \varphi\circ\delta_{\lambda} & \text{pointwise on $\de\Omega_1$},
  \end{cases}
  \end{equation}
  where $\mathcal{Q} v = \sum_{i = 1}^mX_i^2v+\widetilde{Z}v-\widetilde{c}v$ and
  $$\widetilde{c} = \lambda^2(c\circ\delta_\lambda),\qquad
  \widetilde{Z} = \lambda^2\sum_{j = 1}^n\lambda^{-\sigma_j}(b_j\circ\delta_\lambda)\partial_j.$$
  On the other hand, since we know from Lemma \ref{lem:ExistenceOmegazero} that problem \eqref{eq:PbOmegav}
  admits a unique solution $v\in C(\overline{\Omega}_1)$, we conclude that
 problem \eqref{eq:PbOmegaLambdau} admits the unique solution
  $u = v\circ\delta_{\lambda^{-1}}\in C(\overline{\Omega}_\lambda)$.
\end{proof}
\begin{remark} \label{rem:Lemmagenerali}
It is worth noting that the above Lemmas \ref{lem:ExistenceOmegazero}\,-\,\ref{lem:SequenceInvading}
are proved
\emph{for a general H\"ormander operator
$\mathcal{P}$} modeled on the homogeneous vector fields $X_i$'s
(as in \eqref{2}),
\emph{with\-out the need} to assume that the drift $Z$ is of the form
$$Z = \textstyle\sum_{i = 1}^m b_i(x)X_i.$$
\end{remark}
 We are now ready to prove Proposition \ref{prop:Solnonneg}.
\begin{proof}[Proof of Proposition \ref{prop:Solnonneg}]
Let $u$ be a bounded nontrivial solution of \eqref{eq:MainPDE}. If $u$ does not change sign, then it is sufficient to define
\[
\hat{u}=\frac{|u|}{\|u\|_{\infty}},
\]
so that $\hat{u}$ is a solution of \eqref{eq:MainPDE} satisfying $0\leq \hat{u}\leq1$ on $\R^n$. By the Strong Maximum Principle $u>0$ on $\R^n$, and the proof in this case is complete.

Now suppose that $u$ changes sign, so that $u_+,u_-$ are nontrivial. Up to dividing $u$ by $\|u_+\|_{\infty}$, we can suppose without loss of generality that
\[
\sup_{\R^n}u_+=1.
\]
For any $k\in \mathbb N$, let $\mathcal{V}_k$ be as in Lemma \ref{lem:SequenceInvading} and, accordingly, let $v_k$ be the unique distributional solution to problem
\begin{equation*}
\left\{
\begin{array}{ll}
 \,  \mathcal L v = 0  \, &\textrm{in}\,\,\mathcal{V}_k
\\&\\
\textrm{ } v \, = u_+ & \textrm{on}\,\,  \partial \mathcal{V}_k \,\,.
\end{array}
\right.
\end{equation*}
Note that, since $\LL$ is $C^\infty$-hypoelliptic, we have $v_k\in C^\infty(\mathcal{V}_k)\cap C(\overline{\mathcal{V}}_k)$.
Note also that by the Weak Maximum Principle
\begin{equation}\label{e6}
v_k\geq 0 \quad \text{ in }\,\, \mathcal{V}_k\,.
\end{equation}
Furthermore, since $u\leq u^+$ on $\partial\mathcal{V}_k$, applying again the Weak Maximum Principle to $v_k-u$, one has
\begin{equation}\label{e4}
v_k\geq u \quad \text{ in }\,\, \mathcal{V}_k\,.
\end{equation}
From \eqref{e6} and \eqref{e4} we deduce that
\begin{equation}\label{e7}
v_k \geq u_+\quad \text{ in }\,\, \mathcal{V}_k\,.
\end{equation}
Moreover, for any $k\in \mathbb N$, the function $\omega\equiv 1$ satisfies
\[
\left\{
\begin{array}{ll}
 \,  \mathcal L \omega =  - Q \leq 0  \, &\textrm{in}\,\,\mathcal{V}_k
\\&\\
\textrm{ } \omega \, \geq u_+ & \textrm{on}\,\,  \partial \mathcal{V}_k \,\,.
\end{array}
\right.
\]
By the Weak Maximum Principle applied to $v_k-1$, for any $k\in \mathbb N$, we then have
\begin{equation*}
v_k\leq 1 \quad \text{ in }\,\, \mathcal{V}_k\,.
\end{equation*}
We now observe that, in view of \eqref{e7}, for any $k\in \mathbb N$, $v_{k+1}$ satisfies
\[
\left\{
\begin{array}{ll}
 \,  \mathcal L v_{k+1} =  0  \, &\textrm{in}\,\,\mathcal{V}_k
\\&\\
\textrm{ } v_{k+1} \, \geq u_+ & \textrm{on}\,\,  \partial\mathcal{V}_k \,\,.
\end{array}
\right.
\]
Hence, by the Weak Maximum Principle,
\[v_{k+1}\geq v_k\quad \text{ in }\,\, \mathcal{V}_k\,.
\]
Therefore, the sequence $\{v_k\}$ is satisfies
\[
0\leq u^+\leq v_k\leq v_{k+1}\leq1\qquad\text{ in }\mathcal{V}_k.
\]
By monotonicity there exists a pointwise limit $\hat{u}:\R^n\to\R$ of the sequence, such that
\[
0\leq u_+\leq \hat{u}\leq1.
\]
As a consequence $\hat{u}$ is non trivial and, since $v_k$ is a distributional solution of $\LL v_k=0$ in $\mathcal{V}_k$, by the Dominated Convergence Theorem
we have
\[
\LL\hat{u}=0
\]
in the distributional sense in $\R^n$. Since $\LL$ is $C^\infty$-hypoelliptic, $\hat{u}\in C^\infty(\R^n)$ is also a classical solution. Finally, by the Strong Maximum Principle, we conclude that $\hat{u}>0$ in $\R^n$. This ends the proof.
\end{proof}
Now that we have established Proposition \ref{prop:Solnonneg}, we end this section
by proving the fol\-lowing a priori estimate, which will be a key tool
in the proof of Theorem \ref{thm:MainThm}.
Throughout what follows, for every fixed $T > 0$ we set
\[S_T:=\R^n\times [0, T]; \]
moreover, we also define
$$S = \R^n\times[0,+\infty).$$
Finally, when functions $f:S\to\R$ are in\-vol\-ved
(this is the case, e.g., of the functions $v,\,v_+$ and $\zeta$ in
Proposition \ref{prop:intbyparts} below), we understand that
any differential operator built on the vector fields $X_i$'s
\emph{acts in the $x$ variables}.

Hence, for example, we have
$$\nabla_Xf = \nabla_X\big(x\mapsto f(x,t)\big).$$
\begin{proposition}\label{prop:intbyparts}
Let $u$ be a solution of equation \eqref{eq:MainPDE} such that $0\leq u\leq 1$ in $\Omega$. Moreover,
let $T > 0$ be arbitrarily fixed, and let
\begin{equation}\label{e21}
v(x,t):= e^t u(x) -1, \quad \text{ for any }\, (x,t)\in S_T\,.
\end{equation}
Finally, let $\zeta: S_T\to \mathbb R$ and $\varphi:\R^n
\to\mathbb{R}$ be such that
\begin{itemize}
 \item[\emph{(a)}] $\zeta\in \mathrm{Lip}_{\textrm{loc}}(S_T)$ 
 and $\zeta(x, \cdot)\in C^1([0, T])$ for every $x\in \R^n$;
 \vspace*{0.05cm}
 
 \item[\emph{(b)}] $\varphi\in \operatorname{Lip}_{\textrm{loc}}(\R^n)$ be with compact support.
\end{itemize}
Then, there exist absolute constants $\vartheta_1,\,\vartheta_2 > 0$ such that
\begin{equation}\label{eq:intbypartszetaphi}
\begin{aligned}
&\int_{\R^n} Q\varphi^2\big(v_+^2(x,T) e^{\zeta(x,T)}-v_+^2(x,0)e^{\zeta(x,0)}\big)\,dx \\
&\qquad \leq  \iint_{S_T}\left\{Q\zeta_t+|\nabla_X\zeta |^2+\vartheta_1|\mathbf{b}|^2-
 2\mathrm{div}_X(\mathbf{b})\right\}v_+^2 \varphi^2e^{\zeta}\,dx\,dt \\
& \qquad\qquad +\vartheta_2\iint_{S_T}|\nabla_X\varphi|^2 v^2_+e^{\zeta}dx dt\,.
\end{aligned}
\end{equation}
\end{proposition}
\begin{proof}
In view of the definition of $v$ given in \eqref{e21}, since $u$ is by hypothesis a 
so\-lu\-tion of equation \eqref{eq:MainPDE}, we can infer that
\[
Q v_t - \mathcal L_0 v = Q e^t u - e^t \mathcal L_0 u = e^t (Q u -\mathcal L_0 u )=0\quad \text{ in }\, \R^n\,.
\]
This easily yields
\begin{equation}\label{e10}
\iint_{S_T} Q v_t\, v_+\, \varphi^2 e^\zeta dx dt = \iint_{S_T}(\mathcal L_0 v)\, v_+ \varphi^2 e^{\zeta} dx dt\,,
\end{equation}
for any $\varphi$ and $\zeta$ as in the assumptions.

Concerning the first integral in \eqref{e10}, we observe that
\begin{equation}\label{e11}
\begin{aligned}
& \iint_{S_T} Q v_t\, v_+\, \varphi^2 e^\zeta dx dt \\
& \qquad = \int_{\mathbb R^n} Q \varphi^2 \left(\int_0^T v_t\, v_+ e^{\zeta}dt \right) dx
= \int_{\mathbb R^n} Q \varphi^2 \left(\int_0^T \left(\frac 12 v_+^2 \right)_t\, e^\zeta dt \right) dx \\
& \qquad = \frac 12 \int_{\mathbb R^n} Q \varphi^2\big(v_+^2(x,T)e^{\zeta(x, T)}
- v_+^2(x,0)e^{\zeta(x, 0)}\big)dx \\
 &\qquad\qquad - \frac 1 2\iint_{S_T}Q \zeta_t \, v_+^2\varphi^2 e^\zeta dx dt\,.
\end{aligned}
\end{equation}
Concerning the second integral in \eqref{e10}, instead, by
using \eqref{eq:intbyparts},
together with the Lei\-bnitz-ty\-pe formula
\eqref{eq:relationdivnablaX}\,-\,ii), we obtain the following chain of identities:
\begin{equation}\label{e12}
\begin{aligned}
&\iint_{S_T}(\mathcal L_0 v)\, v_+ \varphi^2 e^{\zeta} \,dx dt=
\iint_{S_T}
\sum_{i=1}^m\bigg(X_i^2v+b_i(x)X_iv  \bigg) v_+ \varphi^2 e^\zeta \,dx dt \\
&\qquad=-\iint_{S_T} \langle \nabla_Xv , \nabla_X(v_+ \varphi^2 e^{\zeta }) \rangle\, dx dt \\
& \qquad\qquad -
\iint_{S_T} v\,\mathrm{div}_X(v_+ \varphi^2 e^{\zeta }\,\mathbf{b}) \,dx dt\\
&\qquad=-\iint_{S_T} \left(|\nabla_X v_+|^2 \varphi^2 e^\zeta +2\varphi\,v_+ e^\zeta \langle \nabla_X v_+,\nabla_X \varphi \rangle\right)\,dx dt\\
& \qquad\qquad - \iint_{S_T}v_+ \varphi^2 e^\zeta \langle \nabla_X v_+, \nabla_X \zeta \rangle\,dx\,dt
\\
&\qquad\qquad- \iint_{S_T}v\big(\langle\mathbf{b},\nabla_X(v_+\varphi^2e^{\zeta })\rangle
+ v_+\varphi^2e^{\zeta}\mathrm{div}_X(\mathbf{b})\big)\,dx\,dt\\
& \qquad = -\iint_{S_T} \left(|\nabla_X v_+|^2 \varphi^2 e^\zeta +2\varphi\,v_+ e^\zeta \langle \nabla_X v_+,\nabla_X \varphi \rangle\right)\,dx dt\\
& \qquad\qquad - \iint_{S_T}v_+ \varphi^2 e^\zeta \langle \nabla_X v_+, \nabla_X \zeta \rangle\,dx\,dt
\\
&\qquad\qquad
- \iint_{S_T}\big(v_+\varphi^2e^{\zeta}\langle\mathbf{b},\nabla_Xv_+\rangle
 + 2v_+^2\varphi\,e^{\zeta}\langle\mathbf{b},\nabla_X\varphi\rangle\big)\,dx\,dt \\
 & \qquad\qquad
 -\iint_{S_T}\big(v_+^2\varphi^2 e^\zeta\langle\mathbf{b},\nabla_X\zeta \rangle
 + v_+^2\varphi^2e^{\zeta}\mathrm{div}_X(\mathbf{b})\big)\,dx\,dt \\
 & \qquad = -\iint_{S_T} |\nabla_X v_+|^2 \varphi^2 e^\zeta\,dx\,dt
 -\iint_{S_T}v_+\varphi^2e^{\zeta}\mathrm{div}_X(\mathbf{b})\,dx\,dt \\
 & \qquad\qquad + I_1+I_2+I_3+I_4+I_5,\phantom{\iint_{S_T}}
\end{aligned}
\end{equation}
where we have introduced the following notation
\begin{align*}
I_1&:=-2 \iint_{S_T} \varphi\,v_+ e^\zeta \langle \nabla_X v_+, \nabla_X \varphi\rangle\, dxdt,  \\
I_2&:=-\iint_{S_T}v_+ \varphi^2 e^\zeta \langle \nabla_X v_+, \nabla_X\zeta \rangle\, dx dt, \\
I_3&:=-\iint_{S_T}v_+\varphi^2e^{\zeta}\langle\mathbf{b},\nabla_Xv_+\rangle\,dx\,dt,\\
I_4&:=-2 \iint_{S_T}v_+^2\varphi\,e^{\zeta}\langle\mathbf{b},\nabla_X\varphi\rangle\,dx\,dt,\\
I_5&:=-\iint_{S_T}v_+^2\varphi^2 e^\zeta\langle\mathbf{b},\nabla_X\zeta \rangle\,dx\,dt.
\end{align*}
Now we estimate the integrals $I_i$ for $i=1, \ldots, 5$.
\medskip

-\,\,\emph{Estimate of $I_1$}. For every $\varepsilon_1>0$, we have
\begin{equation}\label{e13}
\begin{aligned}
|I_1|&\leq \iint_{S_T} 2|\varphi| v_+ e^{\zeta}|\nabla_X v_+||\nabla_X \varphi| dx dt\\
& = \iint_{S_T} 2 \left(|\varphi|e^{\frac 12 \zeta}|\nabla_X v_+| \right) \left(v_+ e^{\frac 12 \zeta}|\nabla_X \varphi| \right)\, dx dt\\
&\leq \varepsilon_1 \iint_{S_T} |\nabla_X v_+|^2\varphi^2 e^{\zeta}\, dx dt + \frac 1{\varepsilon_1}\iint_{S_T}|\nabla_X \varphi|^2 v_+^2 e^{\zeta}\, dx dt\,.
\end{aligned}
\end{equation}

-\,\,\emph{Estimate of $I_2$}. For every $\varepsilon_2>0$, we have
\begin{equation*}
\begin{aligned}
|I_2|&\leq \iint_{S_T} v_+ \varphi^2 e^{\zeta}|\nabla_X v_+||\nabla_X \zeta |\,dx dt\\
& = \iint_{S_T}\left(|\varphi|e^{\frac 1 2\zeta}|\nabla_X v_+| \right) \left(v_+ |\varphi|e^{\frac 12 \zeta}|\nabla_X \zeta |\right)\, dx dt\\
&\leq \frac{\varepsilon_2}2\iint_{S_T}|\nabla_X v_+|^2\varphi^2 e^{\zeta}\,dx dt + \frac1{2\varepsilon_2}\iint_{S_T} |\nabla_X \zeta |^2 v_+^2 \varphi^2 e^{\zeta}\, dxdt\,.
\end{aligned}
\end{equation*}

-\,\,\emph{Estimate of $I_3$}. For every $\varepsilon_3>0$, we have
\begin{equation*}
\begin{aligned}
|I_3|& \leq \iint_{S_T}v_+ \varphi^2 e^\zeta |\nabla_X v_+||\mathbf{b}|\,dx dt \\
& = \iint_{S_T} \left(|\varphi|e^{\frac1 2\zeta}|\nabla_X v_+| \right) \left(v_+|\varphi|e^{\frac 1 2\zeta}
|\mathbf{b}|\right)dx\,dt\\
& \leq \frac{\varepsilon_3}{2}\iint_{S_T}|\nabla_X v_+|^2\varphi^2e^{\zeta}\,dx\,dt + \frac1{2\varepsilon_3}\iint_{S_T}|\mathbf{b}|^2 v_+^2 \varphi^2 e^{\zeta}\,dx\,dt.
\end{aligned}
\end{equation*}

-\,\,\emph{Estimate of $I_4$}. We have
\begin{equation*}
\begin{aligned}
|I_4|&\leq 2\iint_{S_T}|\varphi| v_+^2 e^{\zeta}|\nabla_X \varphi||\mathbf{b}|\,dx\,dt \\
& = 2\iint_{S_T}\left(|\varphi|v_+ e^{\frac 1 2\zeta}|\mathbf{b}| \right) \left(v_+ e^{\frac 1 2\zeta}
|\nabla_X\varphi| \right)dx\,dt\\
& \leq \iint_{S_T} |\mathbf{b}|^2 v_+^2 \varphi^2 e^\zeta\,dx\,dt +
\iint_{S_T} |\nabla_X\varphi|^2 v_+^2 e^\zeta\,dx\,dt.
\end{aligned}
\end{equation*}

-\,\,\emph{Estimate of $I_5$.} Finally, for every $\varepsilon_4>0$ we have
\begin{equation}\label{e17}
\begin{aligned}
|I_5|&\leq \iint_{S_T} v_+^2 \varphi^2 e^{\zeta}|\nabla_X \zeta ||\mathbf{b}|\,dx\,dt \\
& =
\iint_{S_T}\left(v_+|\varphi|e^{\frac{1}{2}\zeta}|\nabla_X \zeta | \right)
\left(v_+|\varphi|e^{\frac{1}{2}\zeta}|\mathbf{b}|\right) dx\,dt\\
&\leq \frac{\varepsilon_4}2\iint_{S_T} |\nabla_X\zeta |^2 v_+^2 \varphi^2 e^\zeta\,dx\,dt +
 \frac1{2\varepsilon_4} \iint_{S_T} |\mathbf{b}|^2 v_+^2 \varphi^2 e^\zeta\,dx\,dt.
\end{aligned}
\end{equation}
Gathering \eqref{e13}-to-\eqref{e17}, from \eqref{e12} we then obtain
\begin{equation}\label{e18}
\begin{aligned}
&\iint_{S_T}(\mathcal L_0 v)\, v_+ \varphi^2 e^{\zeta}\,dx\,dt \\
& \qquad \leq \left(-1+\varepsilon_1 +\frac{\varepsilon_2}2+\frac{\varepsilon_3}2 \right)
\iint_{S_T}|\nabla_X v_+|^2\varphi^2 e^\zeta\,dx\,dt
\\
& \qquad\qquad -\iint_{S_T}v_+^2\varphi^2 e^\zeta \mathrm{div}_X(\mathbf{b})\,dx\,dt \\
&\qquad\qquad+\left(1+\frac1{\varepsilon_1}\right)\iint_{S_T}|\nabla_X\varphi|^2 v_+^2 e^\zeta\,dx\,dt \\
& \qquad\qquad +
 \left(\frac{\varepsilon_4}2+\frac 1{2\varepsilon_2}\right)\iint_{S_T}|\nabla_X\zeta |^2 v_+^2
 \varphi^2 e^\zeta\,dx\,dt\\
&\qquad\qquad
 +\left(1+\frac1{2\varepsilon_3}+\frac1{2\varepsilon_4}\right)
 \iint_{S_T}|\mathbf{b}|^2 v_+^2 \varphi^2 e^\zeta\,dx\,dt.
\end{aligned}
\end{equation}
Now we select the parameters $\varepsilon_i>0\,(i=1\ldots, 4)$ in a such a way that
\[\varepsilon_1+\frac{\varepsilon_2}2+\frac{\varepsilon_3}2=1, \quad \frac{\varepsilon_4}2+\frac1{2\varepsilon_2}=\frac 12.\]
With this choice of $\varepsilon_i$, from \eqref{e18} we get
\begin{equation}\label{e19}
\begin{aligned}
&\iint_{S_T}(\mathcal L_0 v)\, v_+ \varphi^2 e^{\zeta}\,dx\,dt
 \\
 & \qquad \leq \iint_{S_T}\left\{\frac 1 2|\nabla_X \zeta |^2 + c_1 |\mathbf{b}|^2 -
  \mathrm{div}_X(\mathbf{b})  \right\}v_+^2\varphi^2 e^\zeta\,dx\,dt
  \\
  &\qquad\qquad +c_2\iint_{S_T}|\nabla_X\varphi|^2 v_+^2 e^\zeta\,dx\,dt,
\end{aligned}
\end{equation}
where
\[c_1=1+\frac1{2\varepsilon_3}+\frac 1{2\varepsilon_4}, \quad c_2=1+\frac1{\varepsilon_1}\,. \]
Finally, by combining \eqref{e11} with \eqref{e19}, from \eqref{e10} we conclude that
\begin{equation*}
\begin{aligned}
&\int_{\R^n} Q(x) \varphi^2(x) v_+^2(x,T) e^{\zeta(x,T)}dx - \int_{\R^n} Q(x)
\varphi^2(x) v_+^2(x,0) e^{\zeta(x,0)}\,dx \\
&
\qquad\qquad\leq\iint_{S_T}\left\{Q\zeta_t+
 |\nabla_X\zeta |^2+2c_1|\mathbf{b}|^2-2 \mathrm{div}_X(\mathbf{b}) \right\}
v_+^2 \varphi^2e^{\zeta}\,dx \\
&\qquad\qquad\qquad +2c_2\iint_{S_T}|\nabla_X\varphi|^2 v^2_+e^{\zeta}\,dx\,dt\,.
\end{aligned}
\end{equation*}
which is precisely the desired \eqref{eq:intbypartszetaphi} (with $\vartheta_i = 2c_i,\,i = 1,2$).
\end{proof}
\section{Proof of Theorem \ref{thm:MainThm}}\label{proof}
 In this section we provide the full proof of Theorem \ref{thm:MainThm}.
 \begin{proof}[Proof (of Theorem \ref{thm:MainThm}).]
 Arguing by contradiction, we suppose that equation
  \eqref{eq:MainPDE} has non-trivial bounded solutions. In particular,
  from Proposition \ref{prop:Solnonneg} we can assume without loss of generality that the solution $u_0$ also satisfies
  \begin{equation} \label{eq:uzeropositive}
   0 < u_0\leq 1\quad\text{pointwise in $\Omega$}.
  \end{equation}
  We then introduce the auxiliary function defined on $S$ $$v(x,t) := e^tu_0(x)-1,$$
  which we claim satisfies the following:
  \begin{equation} \label{eq:claimcentrale}
   v(x,t)\leq 0\quad\text{for every $x\in\mathrm{supp}(Q)$ and $t > 0$}.
  \end{equation}
  Once \eqref{eq:claimcentrale} is established,
  the proof of the theorem will be complete: in fact, taking the limit as $t\to+\infty$ in
  \eqref{eq:claimcentrale}, we obtain
  $$u_0(x)\leq 0\quad\text{for every $x\in\mathrm{supp}(Q)$},$$
  which clearly contradicts \eqref{eq:uzeropositive}, since $Q\not\equiv 0$ in $\R^n$.
  \vspace*{0.05cm}

  Hence, we proceed with the proof of \eqref{eq:claimcentrale}.
  We start by fixing $\tau > 0$
  (to be chosen conveniently small later on) and
  $R>2\rho_0$, and we set
  $$\zeta_R(x,t) := -\frac{1}{2\tau-t}f_R
  (\mathcal{N}(x))^2
  \qquad
  \text{for }(x,t)\in S_\tau,$$
  where 
  the function $f_R:[0,+\infty)\to\R$ is defined as
  $$f_R(r) := \begin{cases}
  \displaystyle\frac{1}{4}\int_{\rho_0}^{r}\sqrt{\hat{q}(t)}\,dt
   & \text{if $r>R$}, \\[0.1cm]
   \\
   \displaystyle  \frac{1}{4}\int_{\rho_0}^R\sqrt{\hat{q}(t)}\,dt & \text{if $r\leq R$}.
   \end{cases}$$
  We also choose $R_1>2R$ and
  $\phi\in \mathrm{Lip}([0,+\infty))$ such that
  \begin{equation} \label{eq:propPhi}
  \begin{split}
   \mathrm{(i)}&\,\,\text{$0\leq \phi\leq 1$ on $[0,+\infty)$}; \\
   \mathrm{(ii)}&\,\,\text{$\phi\equiv 0$ on $[R_1,+\infty)$ and
    $\phi\equiv 1$ on $[0,R]$}.
  \end{split}
  \end{equation}
  Finally we define
  $$\varphi_{R,R_1}:\R^n\to\R,\quad\varphi_{R,R_1}(x) :=
  \phi(\mathcal{N}(x)).$$
  It is easy to see that $\zeta_R$ and
  $\varphi_{R,R_1}$ satisfy the \emph{re\-gu\-la\-ri\-ty assumptions} (a)\,-\,(b)
  in Proposition \ref{prop:intbyparts}; hence by
 \eqref{eq:intbypartszetaphi}, with $\zeta = \zeta_R$ and
  $\varphi = \varphi_{R,R_1}$, we have
  \begin{equation} \label{eq:intbypartsmainThm}
   \begin{aligned}
   &\int_{\R^n} Q(x)\varphi_{R,R_1}^2v_+^2(x,\tau) e^{\zeta_R(x,\tau)}dx -
    \int_{\R^n} Q(x)\varphi_{R,R_1}^2v_+^2(x,0)e^{\zeta_R(x,0)}dx \\
   &\qquad \leq  \iint_{S_\tau}\left\{Q(x)(\zeta_R)_t+|\nabla_X\zeta_R|^2+\vartheta_1|\mathbf{b}|^2-2
   \operatorname{div}_X(\mathbf{b})\right\}v_+^2
   \varphi_{R,R_1}^2e^{\zeta_R}\,dx dt \\
& \qquad\qquad +\vartheta_2\iint_{S_\tau}|\nabla_X\varphi_{R,R_1}|^2 v^2_+e^{\zeta_R}dx dt,
\end{aligned}
\end{equation}
 where $\vartheta_1,\vartheta_2 > 0$ are fixed \emph{universal} constants.

 We will now show that there exists $\tau_0 > 0$, only depending
 on the structural constants appearing in assumption (G), such that
 \begin{equation} \label{eq:keytermnegative}
  Q(x)(\zeta_R)_t+|\nabla_X\zeta_R|^2+\vartheta_1|\mathbf{b}|^2-2
   \operatorname{div}_X(\mathbf{b})\leq 0\quad\text{for a.e.\,$x\in\R^n,\,t\in[0,\tau]$},
 \end{equation}
 for every $\tau\leq \tau_0$. In order to prove \eqref{eq:keytermnegative}, we define
 $$\gamma_{R} := \frac{1}{4}\int_{\rho_0}^R\sqrt{\hat{q}(t)}\,dt.$$
 Note that we can assume $\gamma_R>0$, up to choosing $R$, and hence also $R_1$, sufficiently large. Moreover,
 since the function $\zeta_R$ is differentiable w.r.t.\,$x$ in $\{\mathcal{N}\neq R\}$, inequality
\eqref{eq:keytermnegative} follows if we show that
 $$Q(x)(\zeta_R)_t+|\nabla_X\zeta_R|^2+\vartheta_1|\mathbf{b}|^2-2
   \operatorname{div}_X(\mathbf{b})\leq 0\quad\forall\,\,(x,t)\in\{\mathcal{N}\neq R\}\times[0,\tau].$$
 Now let $(x,t)\in\{\mathcal{N}\neq R\}\times [0,\tau]$, we distinguish two cases.
 \medskip

 \emph{-\,\,Case I:} $\mathcal{N}(x)>R$.  In this case
 we first observe that, by exploiting \eqref{eq:assQgeq}
 (and by taking into account the very definition of $\zeta_R$), we get
 \begin{align*}
    & Q(x)(\zeta_R)_t+|\nabla_X\zeta_R|^2 \\
    & \qquad = -\frac{Q(x)}{(2\tau-t)^2}
   f_R(\mathcal{N}(x))^2
   +\frac{f_R(\mathcal{N}(x))^2}{4(2\tau-t)^2}
   \hat{q}(\mathcal{N}(x))
   |\nabla_X\mathcal{N}(x)|^2
   \\
   & \qquad \leq -\frac{3}{4}|\nabla_X\mathcal{N}(x)|^2\frac{\hat{q}(\mathcal{N}(x))f_R(\mathcal{N}(x))^2}
   {(2\tau-t)^2};
 \end{align*}
 moreover, from \eqref{eq:controllobLontano} (and since $R>2 \rho_0$), we also have
 \begin{align*}
  & \vartheta_1|\mathbf{b}|^2-2
   \operatorname{div}_X(\mathbf{b}) \leq
   C_1\big(|\mathbf{b}|^2+\mathrm{div}_X(\mathbf{b})_-\big)
   \\
   & \qquad\leq C_1\kappa\,|\nabla_X\mathcal{N}(x)|^2\bigg(\int_{\rho_0}^{\mathcal{N}(x)}
   \sqrt{\hat{q}(t)}\,dt
   \bigg)^2\hat{q}(\mathcal{N}(x)) \\
   &\qquad= 16C_1\kappa(2\tau-t)^2\,|\nabla_X\mathcal{N}(x)|^2\frac{\hat{q}(\mathcal{N}(x))f_R(\mathcal{N}(x))^2}
   {(2\tau-t)^2},
 \end{align*}
 for some $C_1 > 0$. Thus, we obtain
 \begin{align*}
   & Q(x)(\zeta_R)_t+|\nabla_X\zeta_R|^2
   +
   \vartheta_1|\mathbf{b}|^2-2
   \operatorname{div}_X(\mathbf{b}) \\
   & \qquad \leq |\nabla_X\mathcal{N}(x)|^2\frac{\hat{q}(\mathcal{N}(x))f_R(\mathcal{N}(x))^2}
   {(2\tau-t)^2}
   \Big(-\frac{3}{4}+16C_1\kappa(2\tau-t)^2\Big) \\[0.05cm]
   & \qquad (\text{recalling that $0\leq t\leq \tau$}) \\[0.05cm]
   & \qquad\leq |\nabla_X\mathcal{N}(x)|^2\frac{\hat{q}(\mathcal{N}(x))f_R(\mathcal{N}(x))^2}
   {(2\tau-t)^2}
   \Big(-\frac{3}{4}+64C_1\kappa\tau^2\Big).
 \end{align*}
 As a consequence, if we choose $\tau_0 > 0$ so small that
 $64 C_1 \kappa\tau_0^2 \leq 3/4$
 (notice that the smallness of $\tau_0$ only depends on $C_1$ and on $\kappa$),
 we conclude that
 $$Q(x)(\zeta_R)_t+|\nabla_X\zeta_R|^2
   +
   \vartheta_1|\mathbf{b}|^2-2
   \operatorname{div}_X(\mathbf{b}) \leq 0,$$
   for every $x\in \Omega$ with $\mathcal{N}(x) > R$ and $0\leq t\leq \tau$,
   provided that $\tau\leq \tau_0$.
   \medskip

   \emph{-\,\,Case II:} $\mathcal{N}(x) < R$.
   In this case, by definition of $\zeta_R$ we infer that
   \begin{equation} \label{eq:estimQzetatnablaCaseII}
    Q(x)(\zeta_R)_t+|\nabla_X\zeta_R|^2 = Q(x)(\zeta_R)_t
   = -\frac{Q(x)}{(2\tau-t)^2}\gamma_R^2.
   \end{equation}
   To estimate the term $\vartheta_1|\mathbf{b}|^2-2
   \operatorname{div}_X(\mathbf{b})$, instead, we
    consider two different sub-ca\-ses, according
   to the growth condition in assumption $(G)$.
   \medskip

   $\mathrm{\emph{-\,\,Case II)}}_a$: $\mathcal{N}(x) >\rho_0$.
   In this case,
   by exploiting once again \eqref{eq:controllobLontano} (and by ta\-king into account the explicit
   definition of $\gamma_{R}$), we get
   \begin{equation} \label{eq:estimbCaseIIsbI}
   \begin{split}
    & \vartheta_1|\textbf{b}|^2-2
   \operatorname{div}_X(\mathbf{b}) \leq
   C_1\kappa\,|\nabla_X\mathcal{N}(x)|^2\Big(\int_{\rho_0}^{\mathcal{N}(x)}\!\!\!
    \sqrt{\hat{q}}\,ds\Big)^2\hat{q}(\mathcal{N}(x)) \\[0.05cm]
   & \qquad (\text{since we are assuming $\mathcal{N}(x) < R$}) \\[0.05cm]
   & \qquad \leq
   16C_1\kappa\,\gamma_{R}^2\,|\nabla_X\mathcal{N}(x)|^2\,\hat{q}(\mathcal{N}(x)) \\
   & \qquad = 16C_1\kappa(2\tau-t)^2\,\gamma_R^2\,|\nabla_X\mathcal{N}(x)|^2
   \frac{\hat{q}(\mathcal{N}(x))}{(2\tau-t)^2}.
   \end{split}
   \end{equation}
   Gathering \eqref{eq:estimQzetatnablaCaseII}-\eqref{eq:estimbCaseIIsbI},
   and using \eqref{eq:assQgeq}, we obtain
   \begin{align*}
    & Q(x)(\zeta_R)_t+|\nabla_X\zeta_R|^2
   +
   \vartheta_1|\textbf{b}|^2-2
   \operatorname{div}_X(\textbf{b}) \\
   & \qquad \leq \gamma_R^2\,|\nabla_X\mathcal{N}(x)|^2\frac{\hat{q}(\mathcal{N}(x))}{(2\tau-t)^2}
   \big(16C_1\kappa(2\tau-t)^2-1\big)
   \\[0.05cm]
   & \qquad (\text{recalling that $0\leq t\leq \tau$}) \\[0.05cm]
   & \qquad\leq \gamma_R^2\,|\nabla_X\mathcal{N}(x)|^2\frac{\hat{q}(\mathcal{N}(x))}{(2\tau-t)^2}
   \big(64C_1\kappa\tau^2-1\big).
   \end{align*}
   As a consequence, if we choose $\tau_0 > 0$ so small that $64C_1\kappa\tau_0^2 \leq 1$
   (notice that, once again, the smallness of $\tau_0$ only
   depends on $C_1$ and $\kappa$), we conclude that
   $$Q(x)(\zeta_R)_t+|\nabla_X\zeta_R|^2
   +
   \vartheta_1|\textbf{b}|^2-2
   \operatorname{div}_X(\textbf{b}) \leq 0,$$
   for every $x\in \R^n$ with $\rho_0<\mathcal{N}(x)<R$ and $0\leq t\leq \tau$,
   provided that $\tau\leq \tau_0$.
   \medskip

    $\mathrm{\emph{-\,\,Case II)}}_b$: $\mathcal{N}(x) \leq \rho_0$.
   In this case, instead, by exploiting
   \eqref{eq:controllobVicino}
   (and by recalling that $R>2\rho_0$), we have the following easier estimate
   \begin{equation} \label{eq:estimbCaseIIsbII}
   \begin{split}
    \vartheta_1|\textbf{b}|^2-2
   \operatorname{div}_X(\textbf{b}) & \leq C_1\kappa 
   Q(x) \\
   & = C_1\kappa'
   \Big(\int_{\rho_0}^{2\rho_0}\sqrt{\hat{q}(t)}\,dt\bigg)^2 Q(x)
   \\
   &  \leq 16C_1\kappa'\gamma_R^2Q(x)\\
   & = 16C_1\kappa'(2\tau-t)^2\gamma_R^2\frac{Q(x)}{(2\tau-t)^2},
   \end{split}
   \end{equation}
   where the constant $\kappa' > 0$ is given by
   $$\kappa' := \kappa\,
   \bigg(\int_{\rho_0}^{2\rho_0}\sqrt{\hat{q}(t)}\,dt\bigg)^{-2}.$$
   We explicitly stress that $\kappa'$ only depends
   on $\kappa,\,\rho_0$ and on the function $q$,
   which are fixed \emph{from the very beginning} in assumption $(G)$.

   Then, gathering \eqref{eq:estimQzetatnablaCaseII} and
   \eqref{eq:estimbCaseIIsbII}, we obtain
   \begin{align*}
    & Q(x)(\zeta_R)_t+|\nabla_X\zeta_R|^2
   +
   \vartheta_1|\textbf{b}|^2-2
   \operatorname{div}_X(\textbf{b}) \\
   & \qquad \leq \gamma_R^2\frac{Q(x)}
   {(2\tau-t)^2}
   \big(16C_1
   \kappa'(2\tau-t)^2-1\big) \\[0.05cm]
   & \qquad (\text{recalling that $0\leq t\leq \tau$}) \\[0.05cm]
   & \qquad\leq \gamma_R^2\frac{Q(x)}
   {(2\tau-t)^2}\big(64C_1\kappa'\tau^2-1\big).
   \end{align*}
   As a consequence, if we choose $\tau_0 > 0$ so small that
   $64C_1\kappa'\tau_0^2 \leq 1$
   (note that the smallness of $\tau_0$ only
   depends on $C_1$ and on the structural constants
   appearing in assumption $(G)$, since the same is true of $\kappa'$), we conclude that
   $$Q(x)(\zeta_R)_t+|\nabla_X\zeta_R|^2
   +
   \vartheta_1|\textbf{b}|^2-2
   \operatorname{div}_X(\textbf{b}) \leq 0,$$
   for every $x\in \R^n$ with $\mathcal{N}(x)\neq R$ and $0\leq t\leq \tau$,
   provided that $\tau\leq \tau_0$.
   \medskip

   Summing up, if we choose $\tau_0 > 0$ so small that
   \begin{equation} \label{eq:choicetauzero}
    64C_1\kappa\tau_0^2 \leq \frac{3}{4}<1\quad\text{and}\quad
   64C_1\kappa'\tau_0^2 \leq 1,
   \end{equation}
   we finally conclude that \eqref{eq:keytermnegative} is satisfied, provided
   that $0\leq\tau\leq\tau_0$.
   With \eqref{eq:keytermnegative} at hand
   we continue with the proof of \eqref{eq:claimcentrale}, assuming from now on that
   $$\text{$0\leq\tau\leq\tau_0$, where $\tau_0$ is as in
   \eqref{eq:choicetauzero}}.$$
   In particular, by combining \eqref{eq:keytermnegative} with
   \eqref{eq:intbypartsmainThm}, we obtain
   \begin{equation} \label{eq:estimdastimare}
    \begin{split}
     &\int_{\R^n} Q(x)\varphi_{R,R_1}^2v_+^2(x,\tau) e^{\zeta_R(x,\tau)}dx -
    \int_{\R^n} Q(x)\varphi_{R,R_1}^2v_+^2(x,0)e^{\zeta_R(x,0)}dx \\
   &\qquad\qquad \leq c\iint_{S_\tau}|\nabla_X\varphi_{R,R_1}|^2 v^2_+e^{\zeta_R}dx dt.
    \end{split}
   \end{equation}
   We then turn to estimate the two sides of \eqref{eq:estimdastimare}.

   As for the left-hand side we observe that,
   owing to the properties of $\phi$ in \eqref{eq:propPhi}
   (and taking into account the
   definition of $\varphi_{R,R_1}$), we immediately get
   $$\text{$\varphi_{R,R_1} \equiv 1$ on $\{x:\,\mathcal{N}(x) \leq R\}$};$$
   from this, since by \eqref{eq:uzeropositive} we have
   $v_+(\cdot,0) = (u_0-1)_+ = 0$ in $\R^n$,
    we obtain
   \begin{equation} \label{eq:estimLHS}
   \begin{split}
    & \int_{\R^n} Q(x)\varphi_{R,R_1}^2v_+^2(x,\tau) e^{\zeta_R(x,\tau)}dx -
    \int_{\R^n} Q(x)\varphi_{R,R_1}^2v_+^2(x,0)e^{\zeta_R(x,0)}dx
    \\
    & \qquad = \int_{\R^n} Q(x)\varphi_{R,R_1}^2v_+^2(x,\tau) e^{\zeta_R(x,\tau)}dx \\
    & \qquad (\text{since $\zeta_R(\cdot,\tau) = -\gamma_R^2/\tau$ on $\{\mathcal{N}\leq R\}$}) \\
    & \qquad \geq e^{-\gamma_R^2/\tau}\int_{\{\mathcal{N}(x)\leq R\}}
    Q(x)\,v_+^2(x,\tau)\,dx.
   \end{split}
   \end{equation}
   As for the right-hand side, instead, we first observe that,
   again by using the proper\-ties of $\phi$ in \eqref{eq:propPhi}, we immediately get
   $$\text{$\nabla_X\varphi_{R,R_1}\neq 0$ only on $\{x:\,R<\mathcal{N}(x)<R_1\}$};$$
   from this, since by \eqref{eq:uzeropositive} we have $-1\leq v(x,t)\leq e^t$,
   we obtain
   \begin{align*}
    & \iint_{S_\tau}|\nabla_X\varphi_{R,R_1}|^2 v^2_+e^{\zeta_R}dx dt  \\
    & \qquad
    = \int_{\R^n}|\nabla_X\varphi_{R,R_1}|^2
    \bigg(\int_0^\tau v_+^2(x,t)\,e^{-f_R(\mathcal{N}(x))^2/(2\tau-t)}\,dt\Big)dx
    \\
    & \qquad
    \leq \tau e^{2\tau}\int_{\{R < \mathcal{N}(x)<R_1\}}
    e^{-f_R(\mathcal{N}(x))^2/(2\tau)}
    |\nabla_X\varphi_{R,R_1}|^2\,dx \\
    & \qquad (\text{recalling that $\varphi_{R,R_1} = \phi\circ\mathcal{N}$}) \\
    & \qquad = \tau e^{2\tau}\int_{\{R < \mathcal{N}(x)<R_1\}}
    e^{-f_R(\mathcal{N}(x))^2/(2\tau)}
    |\phi'(\mathcal{N}(x))|^2|\nabla_X\mathcal{N}(x)|^2\,dx
   =: (\bigstar).
   \end{align*}
   In view of the `$\mathcal{N}$\,-\,radial' expression of the last integral
   in the above estimate, we then use Federer's Coarea formula, obtaining
   \begin{align*}
    (\bigstar) & = \tau e^{2\tau}
    \int_{R}^{R_1}e^{-f_{R}(r)^2/(2\tau)}
  ( \phi'(r))^2\Big(\int_{\{x:\mathcal{N}(x) = r\}}
  \frac{|\nabla_X\mathcal{N}|^2}{|\nabla\mathcal{N}|}\,
  d\mathcal{H}^{n-1}\Big)\,dr \\
  & (\text{recalling the definition of $\mathcal{S}(r)$ in \eqref{eq:defFactorS}}) \\
  & =  \tau e^{2\tau}
    \int_{R}^{R_1}\mathcal{S}(r)\,e^{-f_{R}(r)^2/(2\tau)}
  ( \phi'(r))^2\,dr=: (2\bigstar).
   \end{align*}
  Summing up, we have proved the estimate
  \begin{equation} \label{eq:wherechoosephi}
   \iint_{S_\tau}|\nabla_X\varphi_{R,R_1}|^2 v^2_+e^{\zeta_R}dx dt \leq
   \tau e^{2\tau}
    \int_{R}^{R_1}\mathcal{S}(r)\,e^{-f_{R}(r)^2/(2\tau)}
  ( \phi'(r))^2\,dr,
  \end{equation}
  holding true \emph{for every $\phi\in\mathrm{Lip}([0,+\infty)$} satisfying the properties
  in \eqref{eq:propPhi}.
  To proceed further, we now consider the following specific choice of $\phi$
   $$\phi(r) := -\bigg(\int_{R}^{R_1}
   \frac{e^{f_{R}(z)^2/(2\tau)}}{\mathcal{S}(z)}\,dz\bigg)^{-1}
   \,\int_{R_1}^{r}\frac{e^{f_{R}(z)^2/(2\tau)}}{\mathcal{S}(z)}\,dz$$
   for every $r\in[R,R_1]$, $\phi\equiv0$ in $[R_1,+\infty)$ and $\phi\equiv 1$ in $[0,R]$.

  Clearly, $\phi\in \mathrm{Lip}([0,+\infty))$ and $\phi$ obviously satisfies
  the properties in \eqref{eq:propPhi}; hence, we are entitled to exploit estimate \eqref{eq:wherechoosephi}
  with \emph{this choice of $\phi$}, getting
   \begin{equation} \label{eq:estimRHS}
   \begin{split}
    & \iint_{S_\tau}|\nabla\varphi_{R,R_1}|^2 v^2_+e^{\zeta_R}dx dt
  \leq \tau e^{2\tau}
   \bigg(\int_{R}^{R_1}
   \frac{e^{f_{R}(r)^2/(2\tau)}}{\mathcal{S}(r)}\,dr\bigg)^{-1}.
   \end{split}
   \end{equation}
   All in all, by combining \eqref{eq:estimLHS}\,-\,\eqref{eq:estimRHS} with
   \eqref{eq:estimdastimare} (and by taking into account the
   \emph{explicit e\-xpres\-sion} of $f_{R}(r)$ when $r > R$), we
   obtain
   \begin{equation}  \label{eq:estimToConcludeFin}
   \begin{split}
    & \int_{\{\mathcal{N}(x)\leq R\}}
    Q(x)\,v_+^2(x,\tau)\,dx \\
   & \qquad = c(\tau,R)
   \bigg(\int_{R}^{R_1}
   \frac{1}{\mathcal{S}(r)}
    \exp\bigg\{\frac{1}{32\tau}\Big(\int_{\rho_0}^{r}\sqrt{\hat{q}(t)}\,dt\Big)^2\bigg\}
   \,dr\bigg)^{-1},
  \end{split}
  \end{equation}
  where $c(\tau,R) > 0$ is
  explicitly given by
  $$c(\tau) = c\tau e^{2\tau+\gamma_R^2/\tau}.$$
  With \eqref{eq:estimToConcludeFin} at hand, we finally
  come to the end of the proof. First of all we observe that,
  if $\Lambda > 0$ is as in \eqref{eq:mainAssumptionIntegral},
  by possibly shrinking $\tau_0 > 0$ we may assume that
  $$\frac{1}{32\tau}\geq \frac{1}{32\tau_0}\geq \Lambda$$
  (note that this further shrinking again depends only on \emph{structural constants});
  as a consequence, from estimate \eqref{eq:estimToConcludeFin}  we derive that
  \begin{equation} \label{eq:topasslimite1}
   \begin{split}
   & \int_{\{\mathcal{N}(x)\leq R\}}
    Q(x)\,v_+^2(x,\tau)\,dx \\
    & \qquad \leq c(\tau,R)
   \bigg(\int_{R}^{R_1}
   \frac{1}{\mathcal{S}(r)}
    \exp\bigg\{\Lambda\Big(\int_{\rho_0}^{r}\sqrt{\hat{q}(t)}\,dt\Big)^2\bigg\}
   \,dr\bigg)^{-1}.
  \end{split}
  \end{equation}
  Now, recalling that $R_1>R$ was arbitrarily fixed,
  by taking the limit as $R_1\to +\infty$ in the above
  \eqref{eq:topasslimite1} and by using
  assumption \eqref{eq:mainAssumptionIntegral}, we get
  \begin{equation*}
   \int_{\{\mathcal{N}(x)\leq R\}}
    Q(x)\,v_+^2(x,\tau)\,dx\leq0,
  \end{equation*}
  from which we readily derive that
  $$v(x,\tau)\leq 0\quad\text{for every $x\in\mathrm{supp}(Q)$ with $\mathcal{N}(x)\leq R,\,
  0\leq\tau\leq\tau_0$}.$$
  Then, recalling that also $R>2\rho_0$ was arbitrarily fixed, we
  obtain
  $$v(x,\tau)\leq 0\quad\forall\,\,x\in\mathrm{supp}(Q),\,0\leq\tau\leq\tau_0.$$
  Now, by repeating the \emph{very same argument} exploited so far,
  but integrating on the strip $\R^n\times [\tau_0,\tau_0+\tau]$
  and considering the function
  $$\zeta_R(x,t) := -\frac{1}{2\tau-t+\tau_0}f_R(\mathcal{N}(x))^2,$$
  we can plainly recognize that $v\leq 0$ on $\mathrm{supp}(Q)\times[\tau_0,2\tau_0]$,
  and hence
  $$v(x,\tau)\leq 0\quad\forall\,\,x\in\mathrm{supp}(Q),\,0\leq\tau\leq 2\tau_0.$$
  By iterating this argument, and by exploiting in a crucial way
  the fact that $\tau_0 > 0$
  is a \emph{universal number remaining unchanged at any iteration}, we conclude that
  $$v(x,\tau)\leq 0\quad\forall\,\,x\in\mathrm{supp}(Q),\,\tau \geq 0.$$
  This is precisely the desired \eqref{eq:claimcentrale}, and the proof is finally complete.
  \end{proof}
\section{The case of Carnot groups} \label{sec:CarnotGroups}
 As anticipated in the Introduction, one of the main examples of smooth vector 
 fields satisfying our assumptions $(H1)$-$(H2)$ is given by
 the \emph{generators of the hori\-zon\-tal layer} of a Carnot group 
 (see Appendix \ref{sec:Appendix} for the relevant definitions and few basic facts).
 In this section we rephrase our main Theorem \ref{thm:MainThm} in this context,
 where it takes a simpler form (due to the possibility
 of explicitly computing $\mathcal{S}(r)$).
 
 Let then $\mathbb{G} = (\R^n,*,\delta_\lambda)$ be a Carnot group on $\R^n$, with
 $1\leq m \leq n$ generators; moreover, let $V_1(\G)$
 be the horizontal layer of $\mathrm{Lie}(\G)$, and let
 $$X = \{X_1,\ldots,X_m\}$$
 be a (linear) basis of $V_1(\G)$. Since we have recognized in the Appendix
 that the set $X$ satisfies \emph{both} the assumptions $(H1)$ and $(H2)$ stated
 in the Introduction, we are entitled
 to apply Theorem \ref{thm:MainThm} to the operator 
 $$\LL = \sum_{i = 1}^mX_i^2+\sum_{i = 1}^m b_i(x)X_i-Q(x),$$
 provided that $\mathbf{b} = (b_1,\ldots,b_m)
 \in C^\infty(\R^n;\R^m)$ and $Q\in C^\infty(\R^n)$
 satisfy assumptions $(S)$\,-\,$(G)$ in the statement 
 of this theorem, with respect to a 
 suitable \emph{exhaustion function} $\mathcal{N}:\R^n\to[0,+\infty)$; thus, if we assume that
 $$\int_{\rho_0}^{+\infty}\frac{1}{\mathcal{S}(t)}\exp\Big\{\Lambda
   \Big(\int_{\rho_0}^r\sqrt{\hat{q}(s)}\,ds\Big)^2\Big\}\,dt = +\infty$$
   (here, $\Lambda,\rho_0 >0$ and $\hat{q}$ are as in assumptions $(S)$\,-\,$(G)$),
 we can conclude that $u\equiv 0$ is the unique bounded solution of the equation $\LL u = 0$.
 
 We now aim at showing that, \emph{in the present context}, it is possible to choose
 the exhaustion function $\mathcal{N}$ in such a way that the `geometric' factor
 $$\mathcal{S}(r) = \int_{\{\mathcal{N}(x) = r\}}
 \frac{|\nabla_X\mathcal{N}|^2}{|\nabla\mathcal{N}|}\,d\mathcal{H}^{n-1}$$
 can be \emph{explicitly computed}.
 Indeed, if we consider the operator
 $$\textstyle\Delta_\G = \sum_{i = 1}^m X_i^2$$
 (which is usually referred to as a \emph{sub-Laplacian on $\G$}), we know
 from a deep
 result by Folland \cite{Fo2}  that
 there exists a function
 $$\mathcal{N}_\G:\R^n\to\R$$
 which satisfies the following properties:
  \vspace*{0.1cm}
  
 \begin{compactenum}[i)]
 \item $\mathcal{N}_\G\in C^\infty(\R^n\setminus\{0\})$;
 \vspace*{0.1cm}
 
 \item $\mathcal{N}_\G$ is a \emph{$\delta_\lambda$-homogeneous norm on $\G$}, that is,
 \vspace*{0.1cm}
 
 \begin{compactenum}
  \item $\mathcal{N}_\G\geq 0$ on $\R^n$, and $\mathcal{N}_\G(x) = 0\Leftrightarrow x = 0$;
  \vspace*{0.1cm}
  
  \item for every $x\in\R^n$ and $\lambda > 0$, we have
   $\mathcal{N}_\G(\delta_\lambda(x)) = \lambda\,\mathcal{N}_\G(x)$;
 \end{compactenum}
 \vspace*{0.1cm}
 
 \item denoting by $D\geq n$ the \emph{$\delta_\lambda$-homogeneous dimension of $\G$}, the map
 $$\Gamma(x) = \mathcal{N}_{\G}^{2-D}(x)\qquad (\text{with $x\neq 0$})$$
 is the \emph{unique, global fundamental solution} of $\Delta_\G$, that is,
 \vspace*{0.1cm}
 
 \begin{compactenum}
  \item $\Gamma\in C^\infty(\R^n\setminus\{0\})$;
  \vspace*{0.1cm}
  
  \item $\Gamma\in L^1_{\mathrm{loc}}(\R^n)$ and $\Gamma(x)\to 0$ as $x\to+\infty$;
  \vspace*{0.1cm}
  
  \item for every $\varphi\in C_0^\infty(\R^n)$, we have
   $$\int_{\R^n}\Gamma\,\Delta_\G\varphi\,dx = -\varphi(0).$$
 \end{compactenum}
 \end{compactenum}
 \vspace*{0.1cm}
 
 \noindent Clearly, the function $\mathcal{N}_{\G}$ is an exhaustion function on $\R^n$;
 moreover, owing to pro\-per\-ty ii), it is possible to prove the following
 \emph{surface mean value formula} for $\Delta_\G$ (see, e.g., \cite[Theorem 5.5.4]{BLUlibro}):
 \emph{if $u\in C^\infty(\R^n)$ is such that $\Delta_\G u = 0$ in $\R^n$,
 then for every $x\in\R^n$ and every $r > 0$ we have}
 \begin{equation}\label{eq:MVFGeneral}
  u(x) = \frac{\beta_D}{r^{D-1}}\int_{\{\mathcal{N}_\G(y) = r\}}
 \frac{|\nabla_X\mathcal{N}_\G|^2(x^{-1}*y)}
 {|\nabla_y\mathcal{N}_\G(x^{-1}*y)|}\,d\mathcal{H}^{n-1}(y),
 \end{equation}
 where $\beta_D > 0$ is a constant only depending on $D$. Choosing, in particular,
 $u\equiv 1$ and $x = 0$
 (which is the neutral element of $\G$), from 
 \eqref{eq:MVFGeneral} we then get
 $$\frac{\beta_D}{r^{D-1}}\int_{\{\mathcal{N}_\G(y) = r\}}
 \frac{|\nabla_X\mathcal{N}_\G|^2(y)}
 {|\nabla\mathcal{N}_\G(y)|}\,d\mathcal{H}^{n-1}(y) = 1.$$
 As a consequence, we obtain
 \begin{equation} \label{eq:ExplicitSr}
  \mathcal{S}(r) = \int_{\{\mathcal{N}_\G = r\}}
 \frac{|\nabla_X\mathcal{N}_\G|^2}{|\nabla\mathcal{N}_\G|}\,d\mathcal{H}^{n-1}
 = \beta_D r^{D-1}\quad\text{for every $r > 0$}.
 \end{equation}
 In view of \eqref{eq:ExplicitSr}, if we apply
 Theorem \ref{thm:MainThm} with the choice $\mathcal{N} = \mathcal{N}_\G$,
 we then plainly derive
 the following result in the context of \emph{Carnot groups}.
 \begin{theorem} \label{thm:MainCG}
   Let $\G = (\R^n,*,\delta_\lambda)$ be a Carnot group on $\R^n$,
   with homogeneous dimension $D\geq n$ and $1\leq m\leq n$ generators. Moreover, let $V_1(\G)$
   be the first layer of $\G$, and let
   $X = \{X_1,\ldots,X_m\}$
   be a linear basis of $V_1(\G)$. Setting $\Delta_\G = \sum_{i = 1}^m X_i^2$,
   we denote by $\mathcal{N}_\G$ the unique homogeneous norm on $\R^n$ such that
   $$\Gamma(x) = \mathcal{N}_\G^{2-D}(x)\qquad(x\neq 0)$$
   is the global fundamental solution of $\Delta_\G$, see Folland \cite{Fo2}. Finally, let
   $$\mathbf{b} = (b_1,\ldots,b_m)\in C^\infty(\R^n;\R^m),\qquad Q\in C^\infty(\R^n)$$
   be smooth functions satisfying Assumptions $(S)$ and $(G)$ in the statement 
   of Theo\-rem \ref{thm:MainThm} 
   \emph{(}for some $\kappa,\rho_0 > 0$ and $\hat{q}:(\rho_0,+\infty)\to[0,+\infty)$,
   with respect to $\mathcal{N}_\G$\emph{)}.
   
   Then, if there exists $\Lambda > 0$ such that
   \begin{equation} \label{eq:mainAssumptionIntegralCG}
   \int_{\rho_0}^{+\infty}\frac{1}{r^{D-1}}\exp\Big\{\Lambda
   \Big(\int_{\rho_0}^r\sqrt{\hat{q}(s)}\,ds\Big)^2\Big\}\,dr = +\infty,
 \end{equation}
 the Liouville property \eqref{eq:LPprop} holds for the equation
 $$\textstyle \LL u = \Delta_\G u + \sum_{i = 1}^mb_i(x)X_iu-Q(x)u = 0.$$
 \end{theorem}
\section{Examples}\label{examples}
 In this section we provide some examples
 of application of our main result, that is, Theorem \ref{thm:MainThm}; in order to be
 as concrete as possible, most of these examples will be given in the context
 of the \emph{Carnot groups}, where Theorem \ref{thm:MainThm} takes the more
 explicit form given in Theorem \ref{thm:MainCG}.
 \begin{example} \label{exm:HeisenbergSenzab}
  For a fixed $m\in\mathbb{N}$, 
  let $\mathbb{H}^m$ be the \emph{Heisenberg-Weyl group} on $\R^{2m+1}$: this is
  the Carnot group $\mathbb{H}^m = (\R^{2m+1},*,\delta_\lambda)$ whose
  group opera\-tion $*$ and fa\-mi\-ly of dilations $\{\delta_\lambda\}_\lambda$
  are defined as follows (here and throughout, we de\-no\-te
  the points $z\in\R^{2m+1}$ by $z = (x,y,t)$, with $x,y\in\R^{m}$ and $t\in\R$):
  \begin{align*}
   \mathrm{i)}\,\,&(x,y,t)*(x',y',t') = \Big(x+x',y+y',t+t'+
   \frac{1}{2}\big(\langle y,x'\rangle-\langle x,y'\rangle\big)\Big)\\
   \mathrm{ii)}\,\,&\delta_\lambda(x,y,t) = (\lambda x,\lambda y, \lambda^2 t)
   \quad(\text{hence, $\sigma_1=\cdots=\sigma_{2m} = 1,\,\sigma_{2m+1} = 2$}).
  \end{align*}
  Owing to the above i)\,-\,ii), and taking into account Proposition \ref{prop.LieAlgG},
  it is not difficult to recognize that the \emph{horizontal layer} $V_1(\mathbb{H}^m)$
  of $\mathbb{H}^m$ is spanned by the ($2m$, linearly independent) smooth vector fields
  \begin{equation*}
   \begin{gathered}
   X_i = \de_{x_i}+\frac{1}{2}y_i\de_{t},\qquad Y_i = \de_{y_i}-\frac{1}{2}x_i\de_t
   \qquad (1\leq i\leq m),
   \end{gathered}
  \end{equation*}
  and thus the set $X = \{X_i,Y_i:\,1\leq i\leq m\}$ is a 
  (linear) basis of $V_1(\mathbb{H}^m)$;
  further\-mo\-re, 
  the \emph{$\delta_\lambda$-homogeneous dimension of $\mathbb{H}^m$} is 
  $$D = \textstyle\sum_{i = 1}^{2m+1}\sigma_i = 2m+2.$$
  As in Section \ref{sec:CarnotGroups}, we then set
  $$\textstyle\Delta_{\mathbb{H}^m} = \sum_{i = 1}^m (X_i^2+Y_i^2),$$
  and we denote by
  $\mathcal{N}_{\mathbb{H}^m}$ the unique homogeneous norm on $\R^{2m+1}$ such that 
  $$\Gamma = \mathcal{N}_{\mathbb{H}^m}^{-2m}$$ is the global fundamental solution of
  $\Delta_{\mathbb{H}^m}$ (see Theorem \ref{thm:MainCG}). 
  
  A remarkable fact is that,
  \emph{in this particular
  case},
  such a norm can be \emph{explicitly computed} (see, e.g., \cite[Chapter 5]{BLUlibro}): in fact, we have  
  \begin{equation} \label{eq:explicitNHm}
   \mathcal{N}_{\mathbb{H}^m}(x,y,t) = \mathbf{c}_m\big((|x|^2+|y|^2)^2+16t^2\big)^{1/4},
  \end{equation}
  for a suitable constant $\mathbf{c}_m > 0$
  only depending on $m$. Thus, we also have
  \begin{equation} \label{eq:explicitPsiHm}
   |\nabla_X\mathcal{N}_{\mathbb{H}^m}(x,y,t)|^2 = \mathbf{c}_m^2\cdot\frac{|x|^2+|y|^2}
  {\sqrt{(|x|^2+|y|^2)^2+16t^2}}.
  \end{equation}
  All that being said, we fix $\alpha\leq 2$ and we
  consider the equation
  \begin{equation} \label{eq:SchEqExm1}
   \LL u = \Delta_{\mathbb{H}^m}u-Q_\alpha(x) u = 0
  \qquad\text{on $\mathbb{H}^m \equiv \R^{2m+1}$},
  \end{equation}
  where $Q_\alpha:\R^{2m+1}\to\R$ is a function satisfying the following properties:
  \vspace*{0.1cm}
  
  \begin{compactenum}
   \item $Q_\alpha\in C^\infty(\R^{2m+1}),\,Q_\alpha\geq 0$ on $\R^{2m+1}$;
   \vspace*{0.05cm}
   
   \item setting $\mathcal{O} = \{\mathcal{N}_{\mathbb{H}^m}>1\}\subseteq\R^{2m+1}$, we have
  \begin{equation*} 
   Q_\alpha(z) \geq |\nabla_X\mathcal{N}_{\mathbb{H}^m}(z)|^2
   \mathcal{N}_{\mathbb{H}^m}^{-\alpha}(z)\quad\text{on $\mathcal{O}$}
  \end{equation*}
  (where $\mathcal{N}_{\mathbb{H}^m}$ as in \eqref{eq:explicitNHm}).
  \end{compactenum}
  \vspace*{0.1cm}
  
   Clearly, equation \eqref{eq:SchEqExm1} takes the form
  \eqref{eq:MainPDE}, with 
  $$\text{$\mathbf{b}\equiv 0$ \quad and\quad  $Q \equiv Q_\alpha$ as above};$$
  we then turn to show that \emph{all the assumptions}
  of Theorem \ref{thm:MainCG} are satisfied, so that
  equation \eqref{eq:SchEqExm1} \emph{does not possess}
  non-trivial bounded solutions.
  \medskip

  \noindent -\,\,\textsc{Assumption $(S)$}.
  Owing to the properties (1)\,-\,(2) of $Q_\alpha$ listed above, 
  and taking into account \eqref{eq:explicitNHm}\,-\,\eqref{eq:explicitPsiHm},
   we immediately see that
  Assumption $(S)$ is satisfied.
  \medskip

  \noindent-\,\,\textsc{Assumption $(G)$}.
  First of all, owing once again to property (2) of $Q_\alpha$, we easily see
  that
  estimate \eqref{eq:assQgeq} in Assumption $(G)$-(i) is  satisfied with the choice
  \begin{equation} \label{eq:choicerhoqExm1}
   \mathcal{N} = \mathcal{N}_{\mathbb{H}^m},\quad \rho_0 := 1\quad\text{and}\quad
  \hat{q}(t) = t^{-\alpha}\,\,(t> \rho_0).
  \end{equation}
  Moreover, since $\mathbf{b}\equiv 0$ and $Q_\alpha \geq 0$ pointwise in $\R^{2m+1}$,
  we immediately recogni\-ze that also
  estimate \eqref{eq:controllobLontano} in Assumption $(G)$-(i)
  and Assumption $(G)$-(ii) are satisfied
  (choosing, e.g., $\kappa = 1$).
  \medskip

  Now that we have proved the validity of the \emph{structural assumptions}
  $(S)$\,-\,$(G)$, we turn to prove the validity of
  assumption \eqref{eq:mainAssumptionIntegralCG}: in view of \eqref{eq:choicerhoqExm1},
  and since in this context we have $D = 2m+2$,
  we need to show that there exists $\Lambda > 0$ such that
  \begin{equation} \label{eq:toproveIntegralHm}
   \int_1^{+\infty}
   \frac{1}{r^{2m+1}}\exp\bigg\{\Lambda\Big(\int_1^{r}t^{-\alpha/2}\,dt\Big)^2\bigg\}\,dr = +\infty;
  \end{equation}
   on the other hand, since we are assuming that $\alpha\leq 2$, it is easy
   to check that the above \eqref{eq:toproveIntegralHm}
   is trivially satisfied by choosing $\Lambda = 1$.
 \end{example}
 \begin{example} \label{exm:HmWithb}
  For a fixed $m\in\mathbb{N}$, we let
  $\mathbb{H}^m = (\R^{2m+1},*,\delta_\lambda)$ be
  the Heisenberg-We\-yl gro\-up 
  of order $m$
  introduced in Example \ref{exm:HeisenbergSenzab} (to which we refer for
   \emph{the no\-tation concerning this group}). 
   In particular, denoting the points $z\in\R^{2m+1}$ by
   $$\text{$z = (x,y,t)$, with $x,y\in\R^m$ and $t\in\R$},$$
   we know that the set
   $$X = \Big\{X_i = \de_{x_i}+\frac{y_i}{2}\de_t,\,Y_i = \de_{y_i}-\frac{x_i}{2}\de_t:\,1\leq i\leq m\Big\}$$
   is a basis for $V_1(\mathbb{H}^m)$. 
   We then consider the equation
  \begin{equation} \label{eq:SchordHeisenwithb}
  \begin{split}
   \LL u & = \Delta_{\mathbb{H}^m}u+\langle\mathbf{b}_\beta(z),\nabla_X u\rangle-Q_\alpha(z) u  = 0 \\
   & = \sum_{i = 1}^m (X_i^2+Y_i^2)u + \langle\mathbf{b}_\beta(z),\nabla_X u\rangle-Q_\alpha(z) u  = 0
   \quad \text{on $\mathbb{H}^m = \R^{2m+1}$},
   \end{split}
  \end{equation}
  where $\alpha\leq 2$ and $\alpha-1\leq \beta\leq 2m$ are \emph{fixed} numbers (recall that the
  $\delta_\lambda$\--ho\-mo\-ge\-neous di\-men\-sion of $\mathbb{H}^m$
  is $D = 2m+1$, see Example \ref{exm:HeisenbergSenzab}), 
  while the functions
  $$\mathbf{b}_\beta:\R^{2m+1}\to \R^{2m},\qquad Q_\alpha:\R^{2m+1}\to\R$$ 
  satisfy the following properties (with $\mathcal{N}_{\mathbb{H}^m}$ as in \eqref{eq:explicitNHm}):
  \vspace*{0.1cm}
  
  \begin{compactenum}
   \item[(1)] $Q_\alpha\in C^\infty(\R^{2m+1}),\,Q_\alpha\geq 0$ on $\R^{2m+1}$;
   \vspace*{0.05cm}
   
   \item[(2)] setting $\mathcal{O} = \{\mathcal{N}_{\mathbb{H}^m}>1\}\subseteq\R^{2m+1}$, we have
  \begin{equation} \label{eq:choicebQexm2}
   Q_\alpha(z) = (1+|\nabla_X\mathcal{N}_{\mathbb{H}^m}|^2)
   \mathcal{N}_{\mathbb{H}^m}^{-\alpha}(z)\quad\text{on $\mathcal{O}$};
  \end{equation}
  \item[(1')] $\mathbf{b}_\beta\in C^\infty(\R^{2m+1};\R^{2m})$;
  \item[(2')] $\mathbf{b}_\beta\equiv 0$ on $\{\mathcal{N}_{\mathbb{H}^m}\leq 1\}\subseteq\R^{2m+1}$ and
  \begin{equation} \label{eq:choicebQexm2BIS}
   \mathbf{b}_\beta(z) = \mathcal{N}_{\mathbb{H}^m}^{-\beta}(z)\cdot\nabla_X\mathcal{N}_{\mathbb{H}^m}
   \quad\text{on $\mathcal{O}' = \{\mathcal{N}_{\mathbb{H}^m} > 2\}$}
  \end{equation}
  \end{compactenum}
  \medskip
  
  Clearly,
  equation \eqref{eq:SchordHeisenwithb} takes the form
  \eqref{eq:MainPDE}, with 
  $$\text{$\mathbf{b}\equiv \mathbf{b}_\beta$ \quad and \quad  $Q \equiv Q_\alpha$ as above};$$
  we now turn to show that \emph{all the assumptions}
  of Theorem \ref{thm:MainCG} are satisfied, so that
  equation \eqref{eq:SchordHeisenwithb} \emph{does not possess}
  non-trivial bounded solutions.
  \medskip

  \noindent -\,\,\textsc{Assumption $(S)$}.
  Owing to the properties (1)\,-\,(2) of $Q_\alpha$ listed above, 
  and taking into account \eqref{eq:explicitNHm}\,-\,\eqref{eq:explicitPsiHm},
   we immediately see that
  Assumption $(S)$ is satisfied.
  \medskip

  \noindent-\,\,\textsc{Assumption $(G)$}.
  First of all we observe that, by \eqref{eq:choicebQexm2}, we have
  $$Q_\alpha(z)\geq |\nabla_X\mathcal{N}_{\mathbb{H}^m}|^2
   \mathcal{N}_{\mathbb{H}^m}^{-\alpha}(z)\quad\text{on $\mathcal{O}$};$$
  from this, we immediately derive that
  estimate \eqref{eq:assQgeq} in Assumption $(G)$-(i) is sati\-sfied 
  \emph{for every $\rho_0\geq 1$}, and with the choice (analogous to that in Example
  \ref{exm:GrushinSenzab})
  \begin{equation} \label{eq:choicerhoqExm2}
   \mathcal{N} = \mathcal{N}_{\mathbb{H}^m},\quad 
  \hat{q}(t) = t^{-\alpha}\,\,(t> \rho_0).
  \end{equation}
 Taking into account \eqref{eq:choicebQexm2BIS}, we then \emph{choose} $\rho_0 = 2$
 and we turn to prove the validity of the remaining estimates
 \eqref{eq:controllobLontano}\,-\,\eqref{eq:controllobVicino} in assumption $(G)$
 (note that in Example \ref{exm:HeisenbergSenzab} we have tacitly
 chosen $\rho_0 = 1$ for simplicity, since in that case $\mathbf{b}\equiv 0$).
 \vspace*{0.1cm}
 
  As regards the validity of estimate \eqref{eq:controllobLontano} we first observe that,
  on account of \eqref{eq:choicebQexm2BIS} (and using \cite[Prop.\,5.4.3]{BLUlibro}
  for the computation of $\Delta_{\mathbb{H}^m}(\mathcal{N}_{\mathbb{H}^m})$), we have
  \begin{align*}
   \mathrm{div}_X(\mathbf{b}_\beta) & = -\beta\mathcal{N}_{\mathbb{H}^m}^{-\beta-1}
   |\nabla_X\mathcal{N}_{\mathbb{H}^m}|^2+\mathcal{N}_{\mathbb{H}^m}^{-\beta}\,
   \mathrm{div}_X(\nabla_X\mathcal{N}_{\mathbb{H}^m}) \\
   & = -\beta\mathcal{N}_{\mathbb{H}^m}^{-\beta-1}
   |\nabla_X\mathcal{N}_{\mathbb{H}^m}|^2+\mathcal{N}_{\mathbb{H}^m}^{-\beta}\,
   \Delta_{\mathbb{H}^m}(\mathcal{N}_{\mathbb{H}^m}) \\
   & = \mathcal{N}_{\mathbb{H}^m}^{-\beta-1}
   |\nabla_X\mathcal{N}_{\mathbb{H}^m}|^2(2m-\beta)\quad\forall\,\,z\in\mathcal{O}';
  \end{align*}
  in particular, since we are assuming $\beta\leq 2m$, we see that
  $$\text{$\mathrm{div}_X(\mathbf{b}_\beta)\geq 0$ on $\mathcal{O}'=
  \{\mathcal{N}_{\mathbb{H}^m} > 2\}$}.$$
  On account of this fact, and recalling
  \eqref{eq:choicebQexm2BIS}\,-\,\eqref{eq:choicerhoqExm2}, to prove the validity of 
  estimate \eqref{eq:controllobLontano} we need to show that there exists some $\kappa_1 > 0$ such that
  \begin{equation} \label{eq:toproveGiiExm2r}
   \rho^{-2\beta} \leq \kappa_1
  \bigg(\int_{2}^{\rho}t^{-\alpha/2}\,dt\bigg)^2\rho^{-\alpha}
  \quad\text{for every $\rho >  2$}.
  \end{equation}
  Since we are assuming $\alpha\leq 2$, to prove
  \eqref{eq:toproveGiiExm2r} we distinguish two cases.
  \begin{itemize}
   \item[(1)] $\alpha < 2$. In this first case, a direct computation gives
   $$\bigg(\int_{2}^{\rho}t^{-\alpha/2}\,dt\bigg)^2
   \sim  \eta_\alpha \rho^{2-\alpha}\quad \text{as $\rho\to +\infty$},$$
   where $\eta_\alpha > 0$ is a constant only depending on $\alpha$;
   as a consequence, taking into account that
   $\beta\geq\alpha-1$, we readily obtain
   \begin{align*}
    & \rho^{2\beta}\cdot\bigg\{\bigg(\int_{2}^{\rho}t^{-\alpha/2}\,dt\bigg)^2\rho^{-\alpha}
    \bigg\}\sim \eta_\alpha \rho^{2(1-\alpha+\beta)}\to \ell^* \in(0,+\infty]
   \end{align*}
   as $\rho\to +\infty$, and this immediately implies the desired \eqref{eq:toproveGiiExm2r}.
   \vspace*{0.1cm}
   
   \item[(2)] $\alpha = 2$. In this second case, we have
   $$\bigg(\int_{2}^{\rho}t^{-\alpha/2}\,dt\bigg)^2
   =  \log^2(\rho/2);$$
   as a consequence, since $\beta\geq \alpha-1 = 1$, we obtain
   \begin{align*}
    & \rho^{2\beta}\cdot\bigg\{\bigg(\int_{2}^{\rho}t^{-\alpha/2}\,dt\bigg)^2\rho^{-\alpha}
    \bigg\} = \rho^{2(\beta-1)}\log^2(\rho/2)\to +\infty
    \quad\text{as $\rho\to +\infty$},
   \end{align*}
   and this implies once again the desired \eqref{eq:toproveGiiExm2r}.
  \end{itemize}
   Summing up, we have proved that the needed 
   \eqref{eq:toproveGiiExm2r} holds for some constant $\kappa_1 > 0$
   only depending on the fixed $\alpha\leq 2$, and thus estimate
   \eqref{eq:controllobLontano} is satisfied.
  \vspace*{0.1cm}
  
  As regards the validity of estimate \eqref{eq:controllobVicino}, instead,
  we proceed as follows. On the one hand, since
  $\mathbf{b}_\beta\equiv 0$ on $\{\mathcal{N}_{\mathbb{H}^m} \leq 1\}$
  and since $Q_\alpha\geq 0$ in $\R^{2m+1}$ (see properties
  (1) and (2')), we clearly have the following estimate
  \begin{equation} \label{eq:stimazero}
   0 = |\mathbf{b}_\beta(z)|^2+(\mathrm{div}_X(\mathbf{b}_\beta))_-
  \leq Q_\alpha(z)\quad \text{for every $z\in \{\mathcal{N}_{\mathbb{H}^m} \leq 1\}$};
  \end{equation}
  on the other hand, since by \eqref{eq:choicebQexm2} we have
  $$Q_\alpha(z) = (1+|\nabla_X\mathcal{N}_{\mathbb{H}^m}|^2)\,
   \mathcal{N}_{\mathbb{H}^m}^{-\alpha}(z) > 0\quad\text{on $\mathcal{O}
   =\{\mathcal{N}_{\mathbb{H}^m}>1\}$},$$
  and since $\mathbf{b}_\beta\in C^\infty(\R^{2m+1},\R^{2m})$, we
  easily derive that 
  \begin{equation} \label{eq:stimanonzero}
   |\mathbf{b}_\beta(z)|^2+(\mathrm{div}_X(\mathbf{b}_\beta))_-
  \leq \kappa_2\,Q_\alpha(z)\quad \text{for every $z\in \{1\leq\mathcal{N}_{\mathbb{H}^m} \leq 2\}$},
  \end{equation}
  for some constant $\kappa_2 > 0$. As a consequence, by combining 
  \eqref{eq:stimazero}\,-\,\eqref{eq:stimanonzero}, we derive that
  estimate \eqref{eq:controllobVicino} is satisfied
  on $\{\mathcal{N}_{\mathbb{H}^m} \leq 2\}$, as desired.
  
  Gathering all the above facts,
  we can finally conclude that
  assumption (G) is satisfied in our context with the following choice
  of $\rho_0,\,\hat{q}$ and $\kappa$:
  \begin{equation} \label{eq:choicerhoqExm2Final}
  \rho_0 := 2,\qquad
  \hat{q}(t) = t^{-\alpha}\,\,(t>2),\qquad \kappa := \max\{1,\kappa_1,\kappa_2\}.
  \end{equation}
  
  Now that we have proved the validity of the \emph{structural assumptions}
  $(S)$\,-\,$(G)$, we tu\-rn to prove the validity of
  assumption \eqref{eq:mainAssumptionIntegralCG}: in view of \eqref{eq:choicerhoqExm2Final},
  and since in this context we have $D = 2m+2$,
  we need to show that there exists $\Lambda > 0$ such that
  \begin{equation} \label{eq:toproveIntegralHmwithb}
   \int_1^{+\infty}
   \frac{1}{r^{2m+1}}\exp\bigg\{\Lambda\Big(\int_1^{r}t^{-\alpha/2}\,dt\Big)^2\bigg\}\,dr = +\infty;
  \end{equation}
   on the other hand, since we are assuming that $\alpha\leq 2$, it is easy
   to check that the above \eqref{eq:toproveIntegralHmwithb}
   is trivially satisfied by choosing $\Lambda = 1$.
   \vspace*{0.1cm}
   
   We explicitly point out that the validity of condition 
   \eqref{eq:toproveIntegralHmwithb} \emph{only depends on $\alpha$},
   since the function $\mathbf{b}_\beta$ does not play a role; this is the
   reason why we fixed $\alpha\leq 2$ from the beginning (as in Example \ref{exm:HeisenbergSenzab}),
   even with the presence of a non-zero $\mathbf{b}_\beta$.
 \end{example}
 \begin{example} \label{exm:GrushinSenzab}
 In Euclidean space $\R^2$, we consider the smooth vector fields
 $$X_1 = \de_{x_1},\qquad X_2 = x_1\de_{x_2}.$$
 Since $[X_1,X_2] = \de_{x_2}$, it is readily seen that $X_1,X_2$ satisfy
 the H\"ormander condition \emph{at every point $x\in\R^2$} (and not only at $x = 0$);
 moreover, defining
 \begin{equation} \label{eq:defdeltalambda}
  \delta_\lambda(x) = (\lambda x_1,\lambda^2 x_2)\qquad (x\in\R^2),
 \end{equation}
 a direct computation shows that $X_1,X_2$ are \emph{$\delta_\lambda$-homogeneous of degree $1$}.
 Gathering these facts, we then conclude that the set 
 $X = \{X_1,X_2\}\subseteq\mathcal{X}(\R^2)$
 satisfies \emph{both the structural As\-sump\-tions $(H1)$ and $(H2)$} of Theorem \ref{thm:MainThm}.
 \vspace*{0.1cm}
 
 Taking into account \eqref{eq:defdeltalambda}, we now consider the \emph{exhaustion function}
 \begin{equation} \label{eq:normaNGrushin}
  \mathcal{N}(x) = \big(x_1^4+x_2^2\big)^{1/4}
 \end{equation}
 (which is readily seen to be \emph{$\delta_\lambda$-homogeneous of degree $1$}),
 and we study the fol\-lowing equation, modeled on the above vector fields $X_1,X_2$,
 \begin{equation} \label{eq:GrushinPDE}
 \begin{split}
  \LL u & = (X_1^2+X_2^2)u-Q_\alpha(x) u = 0\\
   & = (\de_{x_1}^2+x_1^2\de_{x_2}^2)u
   -Q_\alpha(x) u = 0 \quad\text{in $\R^2$},
   \end{split}
 \end{equation}
 where $\alpha\in\R$ is fixed, and $Q_\alpha:\R^2\to\R$ is such that
 \vspace*{0.1cm}
  
  \begin{compactenum}
   \item $Q_\alpha\in C^\infty(\R^{2}),\,Q_\alpha\geq 0$ on $\R^{2}$;
   \vspace*{0.05cm}
   
   \item setting $\mathcal{O} = \{\mathcal{N}>1\} = \{x\in\R^2:\,x_1^4+x_2^2 > 1\}\subseteq\R^{2}$, we have
  \begin{equation*} 
   Q_\alpha(x) = \frac{|\nabla_X\mathcal{N}|^2}
   {\mathcal{N}^\alpha(x)}\quad\text{on $\mathcal{O}$}
  \end{equation*}
  (where $\mathcal{N}$ is as in \eqref{eq:normaNGrushin}).
  \end{compactenum}
  \vspace*{0.1cm}
    
     Clearly, equation \eqref{eq:GrushinPDE} takes the form
  \eqref{eq:MainPDE}, with
  \begin{equation*}
   \mathbf{b}\equiv 0\quad\text{and}\quad \text{$Q \equiv Q_\alpha$ as above}.
  \end{equation*}
  We now turn to show that \emph{Assumptions $(S)$\,-\,$(G)$}
  of Theorem \ref{thm:MainThm} are satisfied
  \emph{with respect to the exhaustion function $\mathcal{N}$
  \eqref{eq:normaNGrushin}}; hence,
  equation \eqref{eq:GrushinPDE} \emph{does not possess}
  non-trivial bounded solutions, \emph{provided that
  condition \eqref{eq:mainAssumptionIntegral} holds.} 
  
  In this perspective, we notice that  
  \begin{equation} \label{eq:nablaXnablaN}
   |\nabla_X\mathcal{N}(x)|^2 = \frac{x_1^2(4 x_1^4 + x_2^2)}{4(x_1^4 + x_2^2)^{3/2}},
   \qquad |\nabla\mathcal{N}(x)| =  
    \frac{1}{2}\sqrt{\frac{4x_1^6+x_2^2}{(x_1^4 + x_2^2)^{3/2}}}
  \end{equation}
  Moreover, we stress that
  \emph{in this context we cannot apply Theorem \ref{thm:MainCG}}:
  indeed, since we have
  $(X_2)_0 = 0\in\R^2$ but $X_2\not\equiv 0$,
   by Proposition \ref{prop.LieAlgG} we see
  that \emph{there cannot exist} a Lie group $\G$
  on $\R^2$ such that $X_1,X_2$ are left-invariant on $\G$.
  \medskip

  We then turn to prove the validity of Assumptions $(S)$ and $(G)$.
  \medskip
  
  \noindent -\,\,\textsc{Assumption $(S)$}.
  Owing to the properties (1)\,-\,(2) of $Q_\alpha$ listed above, 
  and taking into account \eqref{eq:nablaXnablaN},
   we immediately see that
  Assumption $(S)$ is satisfied.
  \medskip

  \noindent-\,\,\textsc{Assumption $(G)$}.
  First of all, owing once again to property (2) of $Q_\alpha$, we see that
  estimate \eqref{eq:assQgeq} in Assumption $(G)$-(i) is trivially satisfied with the choice
  \begin{equation*}
   \rho_0 := 1\quad\text{and}\quad
  \hat{q}(t) = t^{-\alpha}\,\,(t> \rho_0)
  \end{equation*}
  (and with respect to $\mathcal{N}$).
  Moreover, since $\mathbf{b}\equiv 0$ and $Q_\alpha \geq 0$ pointwise in $\R^{2}$,
  we immediately recognize that also
  estimate \eqref{eq:controllobLontano} in Assumption $(G)$-(i)
  and Assumption $(G)$-(ii) are satisfied
  (choosing, e.g., $\kappa = 1$).
  \medskip

  Now we have proved the validity of the \emph{structural assumptions}
  $(S)$\,-\,$(G)$, we are entitled to apply our Theorem \ref{thm:MainThm}
  to equation \eqref{eq:GrushinPDE}: \emph{the Liouville property
  \eqref{eq:LPprop} holds for such an equation, provided that there exists $\Lambda > 0$ such that}
  \begin{equation*}
   \int_1^{+\infty}
   \frac{1}{\mathcal{S}(r)}\exp\bigg\{\Lambda\Big(\int_1^{r}t^{-\alpha/2}\,dt\Big)^2\bigg\}\,dr = +\infty,
  \end{equation*}
   \emph{where $\mathcal{S}(r)$ is `explicitly' given by}
   \begin{align*}
    \mathcal{S}(r) & = \int_{\{\mathcal{N}(x) = r\}}
   \frac{|\nabla_X\mathcal{N}(x)|^2}{|\nabla\mathcal{N}(x)|}\,d\mathcal{H}^{n-1}(x) \\
   & = \frac{1}{2r^3}\int_{\{x_1^4+x_2^2 = r^4\}}\frac{x_1^2 (4 x_1^4 + x_2^2)}{\sqrt{4x_1^6+x_2^2}}
   \,d\mathcal{H}^{n-1}(x).
   \end{align*}
 \end{example}
 \section{Optimality} \label{sec:Optimality}
  The aim of this last section is to present
  a concrete setting where the uniqueness result in Theorem \ref{thm:MainThm} \emph{turns
  out to be sharp}.
  More precisely, following the notation in Example \ref{exm:HeisenbergSenzab}, we consider the equation
  \begin{equation} \label{eq:equationSharp}
   \Delta_{\mathbb{H}^m}u - Q_\alpha u = 0 \quad\text{on $\mathbb{H}^m$},
  \end{equation}
  where the smooth and non-negative potential $Q_\alpha$ satisfies 
  $$
   \mathrm{a)}\,\,Q_\alpha(z) = |\nabla_X\mathcal{N}_{\mathbb{H}^m}(z)|^2
   \mathcal{N}_{\mathbb{H}^m}^{-\alpha}(z)\qquad\text{or}\qquad
   \mathrm{b)}\,\,Q_\alpha(z) = 
   \mathcal{N}_{\mathbb{H}^m}^{-\alpha}(z),$$ 
  in the set $\mathcal{O} = \{\mathcal{N}_{\mathbb{H}^m}(z) > 1\}\subseteq\mathbb{H}^m = \R^{2m+1}$.
  
  Then,
  by combining Example \ref{exm:HeisenbergSenzab} with the results of this section
  we will see that the Liouville property
  holds for equation \eqref{eq:equationSharp}
  \emph{if and only if}
 $$\alpha\leq 2,$$
 which is {precisely the range of $\alpha$ for which condition \eqref{eq:mainAssumptionIntegral}
 holds}.
 \vspace*{0.1cm}
 
 Indeed, we have already shown
 in Example \ref{exm:HeisenbergSenzab} that the 
 Liouville property
  holds for equation \eqref{eq:equationSharp} when $\alpha\leq 2$,
  \emph{since condition \eqref{eq:mainAssumptionIntegral} is satisfied}
  (notice that this equation satisfies
  all the structural assumptions of Theorem \ref{thm:MainThm}
  in both cases a)\,-\,b), since $|\nabla_X\mathcal{N}_{\mathbb{H}^m}|^2$ is globally bounded).
  On the other hand, when $\alpha > 2$ (so that
  condition \eqref{eq:mainAssumptionIntegral} \emph{does not hold}), we have the following theorem
  showing that the Liouville property \emph{does not hold} for equation \eqref{eq:equationSharp}
  in both cases. 
  \begin{theorem} \label{thm:Optimal}
    Let $m\geq 2$, and let $Q\in C^\infty(\mathbb{H}^m),\,\text{$Q\geq 0$ on $\mathbb{H}^m$}$
    be such that
    \begin{equation} \label{eq:QleqNExistence}
     Q(z)\leq C\,\mathcal{N}_{\mathbb{H}^m}^{-\alpha}(z)
    \end{equation}
    out of some compact set $K\subseteq\mathbb{H}^m$, and for some constant $C > 0$. 
    Then, if $\alpha > 2$ there exist infinitely many \emph{distinct and bounded solutions} of equation
    \begin{equation}\label{eq:Optimal}
     \Delta_{\mathbb{H}^m}u - Q u = 0 \quad\text{on $\mathbb{H}^m$}.
    \end{equation}
    More precisely, given any $\gamma > 0$ there exists a bounded solution $u_\gamma$
    of the above e\-qua\-tion such that, for every fixed $t\in\R$, one has
    \begin{equation} \label{eq:limituNonUnique}
     u_\gamma(x,y,t)\to \gamma\quad\text{as ${|x|^2+|y|^2}\to+\infty$}.
     \end{equation}
  \end{theorem}
 The proof of Theorem \ref{thm:Optimal} relies on the following lemmas.
 \begin{lemma} \label{lem:Solvability}
  Let $\Omega\subseteq\mathbb{H}^m$ be a bounded open set satisfying the
  following \emph{(in\-trin\-sic) exterior ball condition}: for every $\zeta\in \de\Omega$ there exist
  $z_0\in\mathbb{H}^m,\,r > 0$ such that
  $${\{z:\,\mathcal{N}_{\mathbb{H}^m}(z_0^{-1}* z)\leq r\}}\cap\overline{\Omega} = \{\zeta\}$$
  \emph{(}here, $*$ is the group law $\mathbb{H}^m$ and $z_0^{-1}$ is the inverse of $z_0$
  with respect to $*$\emph{)}. Mo\-re\-o\-ver, let $c\in C^\infty(\mathbb{H}^m)$  be 
  a non-negative
  function.
  \vspace*{0.1cm}
  
  Then, given any $\varphi\in C(\de\Omega)$, the Dirichlet problem
  $$\begin{cases}
  \Delta_{\mathbb{H}^m}u-c(z)u = 0 & \text{in $\Omega$} \\
  u = \varphi & \text{in $\de\Omega$}
  \end{cases}$$
  has a unique solution $u\in C^\infty(\Omega)\cap 
  C(\overline{\Omega})$.
 \end{lemma}
 \begin{proof}
  This result follows from Abstract Potential Theory, see, e.g., \cite[Sec.\,4]{Herve}.
  In particular, the validity of Axioms 1\,-\,to\,-\,3 in \cite[Sec.\,4]{Herve}
  is proved in \cite[Thm.\,8.2]{Bony}, while the validity of
  Axiom 4 follows by choosing the function
  $$P(z) = \mathcal{N}_{\mathbb{H}^m}^{-2m}(z)$$
  as a potential on $\mathbb{H}^m$. Owing to these facts, we can then apply
  the results in the cited \cite[Sec.\,4]{Herve} with the \emph{Bouligand function}
  $$B(z) = \frac{1}{r^{2m}}-P(z_0^{-1}*z).$$
  This ends the proof.
 \end{proof}
 \begin{lemma} \label{lem:Barriera}
 Let $m\geq 2,\,\alpha > 2$. Moreover, let $Q$
 be as in Theorem \ref{thm:Optimal}, and let $R_0 > 0$ 
 be such that $K\subseteq \mathcal{C}_{R_0} := B_{R_0}\times \R$
 \emph{(}here, $B_{R_0}$ denotes the usual Euclidean ball
 with centre $0$ and radius $R_0$ in $\R^{2m}$\emph{)}.
 Then, the function
 $$V(z) = V(x,y,t) =  A\,(x^2+y^2)^{-\beta/2}$$
  satisfies the following properties
  \begin{equation}  \label{eq:propertiesBarrieraV}
  \begin{split}
  \mathrm{i)}\,\,&\text{$\Delta_{\mathbb{H}^m}V\leq -Q$ on $\mathbb{H}^m\setminus \mathcal{C}_{R_0}$}; \\[0.1cm]
  \mathrm{ii)}\,\,&\text{$V > 0$ pointwise on $\mathbb{H}^m\setminus \mathcal{C}_{R_0}$}; \\[0.1cm]
  \mathrm{iii)}\,\,&\text{$V\to 0$ as ${|x|^2+|y|^2}\to +\infty$ (for any fixed $t\in\R$)},
  \end{split} 
  \end{equation}
  provided that $A > 0$ is large enough, and $0<\beta<\min\{2m-2,\alpha-2\}$.
 \end{lemma}
 \begin{proof}
  It is a direct computation, which crucially
  exploits estimate
  \eqref{eq:QleqNExistence} and the as\-sumption
  $\alpha > 2$. We omit any further detail.
 \end{proof}
 Thanks to Lemmas \ref{lem:Solvability}\,-\,\ref{lem:Barriera}, we can now prove Theorem \ref{thm:Optimal}.
  \begin{proof}[Proof (of Theorem \ref{thm:Optimal})]
 We split the proof
into two steps.
\medskip

\textsc{Step I).} In this first step we construct, for every fixed
$\gamma > 0$, a \emph{bounded function} 
$u_\gamma\in C^\infty(\mathbb{H}^m)$ 
which solves \eqref{eq:Optimal}.
To this end we consider, for every fixed $j\in\mathbb{N}$,
the following Dirichlet problem for the operator $L = \Delta_{\mathbb{H}^m}-Q(x)$
\begin{equation}\label{eq711}
\begin{cases}
L u=0 & \text{in}\,\,\Omega_j = B_j\times(-j,j)\subseteq\mathbb{H}^m\\
u = \gamma & \text{in}\,\, \de\Omega_j \\
\end{cases}
\end{equation}
(here, $B_j$ denotes the usual Euclidean ball in $\R^{2m}$ with centre $0$
and radius $j$).

Owing to the assumptions on $Q(z)$, and since the cylinder $\Omega_j$
satisfies the \emph{intrin\-sic exterior ball condition}
in Lemma \ref{lem:Solvability} (as it is convex), we know from the cited
Lemma \ref{lem:Solvability} that problem \eqref{eq711} has a unique solution $u_j\in C^\infty(\Omega_j)
\cap C(\overline{\Omega}_j)$.

We then claim that the following facts hold.
\begin{itemize}
\item[(1)] For every $j\in\mathbb{N}$ we have
\begin{equation}\label{eq712}
0\le u_j(z)\le \gamma\quad \text{for any}\,\,\,z\in \overline{\Omega}_j.
\end{equation}
\item[(2)] The sequence $\{u_j\}_j$ is \emph{decreasing}.
\end{itemize}

\noindent -\,\,\emph{Proof of Claim} (1). On the one hand, 
since $u_j \in C^\infty(\Omega_j)
\cap C(\overline{\Omega}_j)$ satisfies
\eqref{eq711} and since $\gamma> 0$,
from the Weak Maximum Principle for $L$ (see Remark \ref{rem:PropertiesValgonoGenerali})
we immediately derive that
$u_j(z) \geq 0$ for every $z\in\overline{\Omega}_j$.

On the other hand, setting $w_j = u_j-\gamma
\in C^\infty(\Omega_j)
\cap C(\overline{\Omega}_j)$, again by
exploiting the fact that
$u_j$ solves \eqref{eq711} (and since $Q\geq 0$) we derive that
\begin{itemize}
 \item $Lw_j = Q\gamma\geq 0$ on $\Omega_j$;
 \vspace*{0.05cm}
 
 \item $w_j(\zeta) = u_j(\zeta)-\gamma = 0$ for every $\zeta\in\de\Omega_j$.
\end{itemize}
Gathering these facts, we can exploit once again
the
Weak Maximum Principle for $L$, obtaining $w_j = u_j-\gamma\leq 0$ on $\overline{\Omega}_j$.
Hence, Claim (1) is proved.
\vspace*{0.1cm}

\noindent-\,\,\emph{Proof of Claim} (2). We apply once again the Weak Maximum Principle
for $L$. First of all, since 
$u_j\in C^\infty(\Omega_j)
\cap C(\overline{\Omega}_j)$ solves problem \eqref{eq711}, setting 
$$w_j = u_{j+1}-u_j\in C^\infty(\Omega_j)
\cap C(\overline{\Omega}_j)$$ 
 we  have (notice that, by definition, $\Omega_j\subseteq\overline{\Omega}_j
\subseteq\Omega_{j+1}$)
$$
L w_j(z) = 0\quad\text{for every $z\in\Omega_j$}.
$$
Moreover, on account of \eqref{eq712} we also get
$$\text{$w_j(\zeta) = u_{j+1}(\zeta)-u_j(\zeta) = u_{j+1}(\zeta)-\gamma\leq 0$ 
for all $\zeta\in \de\Omega_j\subseteq\Omega_{j+1}$}.
 $$
Therefore, by the Weak Maximum Principle we conclude that $w_j\leq 0$ in $\overline{\Omega}_j$ 
and, in particular, for every $j\in\mathbb{N}$ we conclude that
\begin{equation*}
u_{j+1}\leq u_j\quad \text{in}\,\,\,\overline{\Omega}_j.
\end{equation*}
This completes the proof of Claim (2).
\vspace*{0.1cm}

Now, by combining Claim (1) and Claim (2) we deduce that the sequence $\{u_j\}_{j\in\mathbb{N}}$ is \emph{decreasing and bounded} on $\overline{\Omega}_j$; this, together with the fact that 
$$\textstyle\bigcup_{j\geq 1}\Omega_j = \mathbb{H}^m,$$ 
ensures that there exists 
 $u_\gamma:\mathbb{H}^m\to\R$ such that
\begin{equation} \label{eq:propugammaBd}
 \begin{split}
 \mathrm{i)}\,\,&\text{$u_\gamma(z) = \lim_{j\to+\infty}u_j(z)$ for every $z\in\mathbb{H}^m$} \\[0.1cm]
 \mathrm{ii)}\,\,&\text{$0\leq u_\gamma(z)\leq \gamma$ for every $z\in\mathbb{H}^m$}.
 \end{split}
\end{equation}
We then turn to show that $u_\gamma\in C^\infty(\mathbb{H}^m)$ and that $u_\gamma$ solves
equation \eqref{eq:Optimal}.

To this end we first observe that, by construction, $u_\gamma\in L^\infty(\mathbb{H}^m)$
is a \emph{distributio\-nal solution} of \eqref{eq:Optimal}:
in\-deed, given any test function $\varphi\in C_0^\infty(\mathbb{H}^m)$, if $j$ is sufficien\-tly large
(so that $\Omega_j\supseteq \mathrm{supp}(\varphi)$) we have
$$0 = \int_{\mathbb{H}^m}(Lu_j)\varphi\,dz = \int_{\mathbb{H}^m}u_j(\Delta_{\mathbb{H}^m}\varphi-
Q(z)\varphi)\,dz,$$
since $u_j\in C^\infty(\Omega_j)\cap C(\overline{\Omega}_j)$ is a 
(classical) solution
of \eqref{eq711}; from this, by a standard do\-mi\-nated\,-\,convergence argument
based on property ii) we get
$$0 = \lim_{j\to+\infty} 
\int_{\mathbb{H}^m}u_j(\Delta_{\mathbb{H}^m}\varphi-
Q(z)\varphi)\,dz = 
\int_{\mathbb{H}^m}u_\gamma(\Delta_{\mathbb{H}^m}\varphi-
Q(z)\varphi)\,dz,$$
and this proves that $Lu_\gamma = 0$ in $\mathcal{D}'(\mathbb{H}^m)$ (by the arbitrariness of $\varphi$).

On account of this fact, 
and since $L$ is $C^\infty$\,-\,hypoelliptic (by H\"ormander's theo\-rem,
see Remark \ref{rem:PropertiesValgonoGenerali}), we
thus conclude that $u_\gamma\in C^\infty(\mathbb{H}^m)$ (up to modifying such a 
function on a set with zero Lebesgue measure), and therefore
$$Lu_\gamma = 0\quad\text{pointwise on $\mathbb{H}^m$}.$$

\textsc{Step II).} In this second step we prove that
 the function $u_\gamma\in C^\infty(\mathbb{H}^m)$, which we
know to be a solution of equation \eqref{eq:Optimal}, satisfies also
\eqref{eq:limituNonUnique}. 

To this end we  de\-fi\-ne, for every $j\in\mathbb{N},\,j > R_0$,
$$w_j(z) = u_j(z) - \gamma+\delta V(z)\in C^\infty(\Omega_j\setminus\mathcal{C}_{R_0})$$ 
(where $u_j\in C^\infty(\Omega_j)\cap C(\overline{\Omega}_j)$ is the
unique solution of the Di\-ri\-chlet problem \eqref{eq711},
while $V$ and $R_0$ are as in Lemma \ref{lem:Barriera}); then, we claim that 
\begin{equation} \label{eq:wleqzeroBarriera}
\text{$w_j\geq 0$ pointwise on $\Omega_j\setminus \mathcal{C}_{R_0}$}, 
\end{equation}
provided that the constant $\delta > 0$ is properly chosen. 

To prove this claim, it suffices to apply the Weak Maximum Principle
for $L$ to the fun\-ction $w_j$
on $\Omega_j\setminus\mathcal{C}_{R_0}$. Indeed, owing to \eqref{eq:propertiesBarrieraV} 
(and recalling that $u_j$ solves
problem \eqref{eq711}), we have the following computations:
\begin{align*}
 \mathrm{i)}\,\,&L w_j(z) = Q(z)\gamma+\delta LV(z) \leq Q(z)(\gamma-\delta)
 \quad\text{for all $z\in\Omega_j\setminus\mathcal{C}_{R_0};$} \\
 \mathrm{ii)}\,\,&w_j = \delta V\geq 0
 \quad\text{on $\zeta\in \big(\de B_j\times[-j,j]\big)\cup \big((\overline{B}_j
 \setminus B_{R_0})\times \{\pm j\}\big)\subseteq \de\Omega_j$}; \\
 \mathrm{iii)}\,\,&w_j(\zeta)\geq -\gamma+\delta V(R_0)\quad\text{for all $\zeta
 \in \de B_{R_0}\times [-j,j]$}.
\end{align*}
 We explicitly notice that, in point iii), we have also used \eqref{eq712}.
 \vspace*{0.1cm}
 
 In view of these facts, if we choose $\delta > 0$ in such a way that
 $$1)\,\,\delta\geq \gamma,\qquad 2)\,\,\delta\geq \gamma/V(R_0),$$
 we can apply the Weak Maximum Principle
 for $L$, thus obtaining \eqref{eq:wleqzeroBarriera}.
 \medskip
 
 Now we have established  \eqref{eq:wleqzeroBarriera}, we can easily conclude the proof
 of \eqref{eq:limituNonUnique}. Indeed, owing to the cited  \eqref{eq:wleqzeroBarriera}, and letting
 $j\to+\infty$, we derive that 
 $$u_\gamma(z) = \lim_{j\to+\infty}u_j(z) \geq\gamma-\delta V(z)\quad
 \text{for all $z\in \mathbb{H}^m\setminus\mathcal{C}_{R_0}$}.$$
 From this, since we have already recognized that $u_\gamma\leq \gamma$ on $\mathbb{H}^m$
 (see \eqref{eq:propugammaBd}\,-\,ii)), by letting ${|x|^2+|y|^2}\to+\infty$ with the help
 of \eqref{eq:propertiesBarrieraV} we conclude that
 $$\text{$u_\gamma(x,y,t_0)\to \gamma$ as ${|x|^2+|y|^2}\to+\infty$}.$$
 This ends the proof.
\end{proof}
\begin{remark} \label{rem:CasoRadiale}
 In the particular case of equation \eqref{eq:Optimal} with
 $$Q_\alpha(z) \leq |\nabla_X\mathcal{N}_{\mathbb{H}^m}(z)|^2
   \mathcal{N}_{\mathbb{H}^m}^{-\alpha}(z),$$
   one can demonstrate that the Liouville property holds \emph{if and only if} $\alpha\leq 2$
   also when $m = 1$ (so that $D = 4$). The argument is very similar to the one detailed above,
   the unique difference being the need to use a barrier of the form
   $$V(z) = A\,\mathcal{N}_{\mathbb{H}^m}^{-\beta}(z)$$
    (in place of the one constructed
    in Lemma \ref{lem:Barriera}),
   with $A$ sufficiently large and 
   $$0<\beta < \min\{\alpha-2,2\},$$ provided
   that $\alpha > 2$. In such a case, for every $\gamma > 0$
   one can construct a solution $u_\gamma$ of \eqref{eq:Optimal} such that
   $$\text{$u(z)\to\gamma$ as $\mathcal{N}_{\mathbb{H}^m}(z)\to+\infty$}.$$
   We leave the details to the interested reader.
\end{remark}
\appendix
 \section{A brief review on Carnot groups}  \label{sec:Appendix}
 In this appendix we review
 the basic definitions and facts on homogeneous
 Carnot groups which have been used in the paper.
 We follow the e\-xpo\-si\-tion in \cite{BLUlibro}, to which
 we refer for all the omitted proofs and the details. \medskip

 To begin with, we recall that $\G = (\R^n,*)$
 is a \emph{Lie group} (on $\R^n)$ if 
 \vspace*{0.1cm}
 
 \begin{compactenum}
  \item $(\R^n,*)$ is an algebraic group, with neutral element $e$;
  \vspace*{0.05cm}
  
  \item the map $*$ is smooth on $\R^n\times\R^n$.
 \end{compactenum} 
 \vspace*{0.1cm}
 
 \noindent For every fixed $x\in\G\equiv\R^n$ we denote, respectively,
 by $\tau_\alpha$ and $\rho_\alpha$ the left-tran\-slation
 and the right-translation by $x$, on $\G$, that is,
 \vspace*{0.1cm}
 
 \begin{compactenum}
  \item[(a)] $\tau_x:\R^n\to\R^n, \qquad\tau_x(y) := x*y$; 
  \vspace*{0.05cm}
  
  \item[(b)] $\rho_x:\R^n\to\R^n, \qquad\rho_x(y) := y*x$.
 \end{compactenum}
 \vspace*{0.1cm}

 \noindent 
 Given a Lie group $\G = (\R^n,*)$ and 
 a smooth vector field $Z$
 on $\R^n$,
 we say that $Z$ is
 \emph{left-invariant on $\G$} if we have
 \begin{equation} \label{eq.defXleftinv}
  Z\big(u\circ\tau_\alpha) = (Zu)\circ\tau_\alpha,
 \end{equation}
 for every $u\in C^\infty(\R^n)$ and every fixed $\alpha\in\R^n$.
 We denote by $\LieG$ the set of the left\--in\-va\-riant vector fields
 on $\G$ and we call it the \emph{Lie algebra of $\G$}. 
 
 Owing to identity \eqref{eq.defXleftinv}, it is quite easy
 to prove the following result.
 \begin{proposition} \label{prop.LieAlgG}
  Let $\G = (\RN,*)$ be a Lie group with neutral
  element $e$,
  and let $\LieG$ be the Lie algebra of $\G$. 
  Then the following facts hold.
  \vspace*{0.1cm}
  
  \begin{compactenum}[(1)]
   \item $\LieG$ is a Lie algebra and $\dim_{\R}(\LieG) = n$;
   \vspace*{0.05cm}
   
   \item for every $i \in \{1,\ldots,N\}$ 
   there exists precisely
   one $J_i\in \LieG$ such that
   $$J_iu(e) = \frac{\de u}{\de x_i}(e) \qquad\big(\text{for every
   $u\in C^\infty(\R^n)$}\big);$$
   more precisely, we have
   $$J_i = \sum_{j = 1}^N a_{ij}(x)\,\frac{\de }{\de x_j},$$
   where $\mathcal{J}_{\tau_x}$ denotes the Jacobian
   matrix of $\tau_x$ at $y = e$, and 
   $(a_{i1}(x),\ldots, a_{iN}(x))^T$ is the $i$-th column
   of the matrix 
   $\mathcal{J}_{\tau_x}(e)$.
   \vspace*{0.05cm}
   
   \item If $J_1,\ldots,J_N$ are as above, then $\mathcal{J}
   = \{J_1,\ldots,J_N\}$ is a 
   {(}linear{)} basis of
   $\LieG$, which is usually called the
   \emph{Jacobian basis of $\LieG$}.
  \end{compactenum}
 \end{proposition}
 Let now $\G = (\R^n,*)$ be a Lie group on $\R^n$, and
 let us assume that there exist
 \emph{natural numbers $1 = \sigma_1\leq\sigma_2\leq\ldots\leq \sigma_n$} such that
 \begin{equation} \label{eq.formDlambda}
  \delta_\lambda:\R^n\to\R^n,\qquad \delta_\lambda(x) :=
 (\lambda^{\sigma_1}x_1,\ldots,\lambda^{\sigma_n}x_n),
 \end{equation}
 is an automorphism of $\G$ for every
 fixed $\lambda > 0$, that is,
 $$\delta_\lambda(x*y) = \delta_\lambda(x)*\delta_\lambda(y), \qquad\text{for every
 $x,y\in\R^n$}.$$
 Then, the triplet
 $\G = (\R^n,*,\delta_\lambda)$ is called a \emph{homogeneous
 (Lie) group} (on $\R^n$), and
  $\{\delta_\lambda\}_{\lambda > 0}$ is usually
 referred to as the \emph{family of dilations}
 of $\G$. The number
 $$\textstyle D = \sum_{i = 1}^n \sigma_i\geq n$$
 is called the \emph{$\delta_\lambda$-homogeneous dimension} (of $\G$).
 \medskip
 
 Let $\G = (\R^n,*,\delta_\lambda)$ be a homogeneous group 
 on $\R^n$
  (with $\delta_\lambda$ as in \eqref{eq.formDlambda})
 and let $\mathcal{J}
 = \{J_1,\ldots,J_N\}$
 be the Jacobian basis of $\LieG$. It is very
 easy to recognize that $J_i$
 is $\delta_\lambda$-homogeneous of degree $\sigma_i$
 (for every fixed $i = 1,\ldots,n$), that is,
 \begin{equation*}
 J_i(u\circ \delta_\lambda) = \lambda^{\sigma_i}\,(J_i u)\circ \delta_\lambda,
 \end{equation*}
 for every $u\in C^\infty(\R^n,\R)$ and every
 $\lambda > 0$; we then define the space
 \begin{equation*}
 V_1(\G) := \mathrm{span}\big\{J_i:\,\sigma_i = 1\big\}\subseteq
 \LieG,
 \end{equation*}
 and we call such a space the  \emph{horizontal layer}
 of $\G$.
 \begin{remark} \label{rem.onV1Gdausare}
  Let $\G = (\R^n,*,\delta_\lambda)$ be a homogeneous group and let
  $V_1(\G)$ be the horizontal layer of $\G$. Since
  $V_1(\G)$ is spanned by vector fields which are
  $\delta_\lambda$-ho\-mo\-ge\-neous of degree $1$
  it is straightforward to recognize that
  $$X\in V_1(\G)\,\,\Longrightarrow\,\,
  \text{$X$ is $\delta_\lambda$-homogeneous
  of degree $1$}.$$
 \end{remark}
 If $\G = (\R^n,*,\delta_\lambda)$ is a homogeneous
 Lie group on $\R^n$ and if
 \begin{equation} \label{eq:condCGroup}
  \mathrm{Lie}\big(V_1(\G)\big) = \LieG,
 \end{equation}
 we say that $\G$ is a (homogeneous) \emph{Carnot group}
 (here, by $\mathrm{Lie}(V_1(\G))$ we mean the smallest Lie sub-algebra
 of $\LieG$ containing $V_1(\G)$).
 The number
 $$m = m(\G) := \dim_{\R}\big(V_1(\G)\big)
 = \mathrm{card}\big\{i\in\{1,\ldots,N\}:\,\sigma_i = 1
 \big\},$$
 is usually referred to as the \emph{number of generators}
 of $\G$. 
 \vspace*{0.1cm}
 
 Let now $X = \{X_1,\ldots,X_m\}$ be \emph{any fixed} (linear) basis of $V_1(\G)$.
 It is not difficult to recognize that the Lie-generating condition
 \eqref{eq:condCGroup}, jointly with the fact that $\mathrm{dim}(\LieG) = n$, ensures that
 \emph{$X_1,\ldots,X_m$ satisfy H\"ormander's condition at every point $x\in\R^n$};
 thus, since the $X_i$'s are also \emph{$\delta_\lambda$-ho\-mo\-ge\-neous of degree $1$}
 (as they belong to $V_1(\G)$, see Remark \ref{rem.onV1Gdausare}), we conclude that
 $$\text{\emph{$X_1,\ldots,X_m$ satisfy Assumptions $(H1)$ and $(H2)$ in the Introduction}}.$$
 The second-order
 differential operator
 $$\textstyle\Delta_\G := \sum_{j = 1}^m X_j^2,$$
 is called a \emph{sub-Laplacian} on $\G$.

\end{document}